\documentclass{article}

\usepackage[usenames,dvipsnames]{xcolor}
\usepackage{amsfonts}
\usepackage{amsmath,amsthm,amssymb,dsfont}
\usepackage{enumerate}
\usepackage[english]{babel}
\usepackage{graphicx}	
\usepackage[caption=false]{subfig}
\usepackage[margin=3cm]{geometry}
\usepackage{url}
\usepackage{todonotes}
\usepackage{bbm}
\usepackage{booktabs}

\usepackage{tikz}
\definecolor{myyellow}{RGB}{242,226,149}
\definecolor{grey}{RGB}{150,150,150}
\definecolor{myblue}{RGB}{200,220,230}

\usepackage{pgflibraryarrows}		

\usepackage{pgfplots}
\pgfplotsset{compat=newest}

\usepackage{mathrsfs}

\usepackage{hyperref}
\hypersetup{colorlinks=true,citecolor=blue,linkcolor=blue,filecolor=blue,urlcolor=blue,breaklinks=true}

\usepackage{nicefrac}
\usepackage{mathtools}

\theoremstyle{plain}
\newtheorem{theorem}{Theorem}[section]
\newtheorem{lemma}[theorem]{Lemma}

\newtheorem{corollary}[theorem]{Corollary}
\newtheorem{proposition}[theorem]{Proposition}

\theoremstyle{definition}
\newtheorem{definition}[theorem]{Definition}
\newtheorem{remark}[theorem]{Remark}
\newtheorem{example}[theorem]{Example}

\newcommand*{\cA}{\mathcal{A}}

\newcommand*{\cD}{\mathcal{D}}

\newcommand*{\RR}{\mathbb{R}}

\newcommand*{\CC}{\mathbb{C}}

\newcommand*{\eps}{\varepsilon}
\newcommand*{\diag}{\mathrm{diag}}

\DeclareMathOperator{\tr}{tr}

\renewcommand*{\>}{\rangle}

\newcommand{\psd}{\succeq}

\DeclareMathOperator*{\argmin}{argmin}

\definecolor{mycolor1}{rgb}{0.00000,0.44700,0.74100}%
\definecolor{mycolor2}{rgb}{0.85000,0.32500,0.09800}%
\definecolor{mycolor3}{rgb}{0.92900,0.69400,0.12500}%
\definecolor{mycolor4}{rgb}{0.49400,0.18400,0.55600}%

\renewcommand{\H}{\mathbf{H}}
\renewcommand{\S}{\textbf{H}}

\newcommand{\Ybar}{Y_s}
\newcommand{\Vbar}{V_s}
\newcommand{\Bbar}{B_s}

\newcommand{\nsd}{\preceq}
\newcommand{\pd}{\succ}

\DeclareMathOperator{\dom}{{\bf dom}}
\DeclareMathOperator{\epi}{{\bf epi}}
\DeclareMathOperator{\cl}{{\bf cl}}
\DeclareMathOperator{\hypo}{{\bf hypo}}
\DeclareMathOperator{\interior}{{\bf int}}

\newcommand{\ones}{\boldsymbol{1}}

\newcommand{\e}{\psi}
\newcommand{\p}{\Psi}

\title{Optimal self-concordant barriers for quantum relative entropies}

\author{Hamza Fawzi\thanks{Department of Applied Mathematics and Theoretical Physics, University of Cambridge, Cambridge CB3 0WA, United Kingdom. \texttt{h.fawzi@damtp.cam.ac.uk}} \and James Saunderson\thanks{Department of Electrical and Computer Systems Engineering, Monash University, Victoria 3800, Australia. \texttt{james.saunderson@monash.edu}}}
\date{\today}

\begin{document}

\maketitle

\begin{abstract}

	Quantum relative entropies are jointly convex functions of two positive
	definite matrices that generalize the Kullback-Leibler divergence and
	arise naturally in quantum information theory.  In this paper, we prove
	self-concordance of natural barrier functions for the epigraphs of
	various quantum relative entropies and divergences. Furthermore we show
	that these barriers have optimal barrier parameter. These barriers
	allow convex optimization problems involving quantum relative entropies
	to be directly solved using interior point methods for non-symmetric
	cones, avoiding the approximations and lifting techniques used in
	previous approaches.  More generally, we establish the self-concordance
	of natural barriers for various closed convex cones related to the
	noncommutative perspectives of operator concave functions, and show
	that the resulting barrier parameters are optimal.	
\end{abstract}


\section{Introduction}
\label{sec:intro}

Given a pair of Hermitian positive definite matrices $X,Y$, the (Umegaki) quantum relative entropy $D(X|Y)$ is defined as
\begin{equation}
\label{eq:qre}
D(X|Y) = \tr(X \log X - X \log Y),
\end{equation}
where $\log$ denotes the matrix logarithm. The quantum relative entropy is a divergence measure between positive definite matrices that plays an important role in quantum information theory, as well as in other areas such as learning theory \cite{kulis2009low,chandrasekaran2017relative,bach2022information}. A fundamental property about $D(X|Y)$ is that it is jointly convex in $(X,Y)$, a property first established by Lieb and Ruskai \cite{liebruskai} building on an earlier result of Lieb \cite{lieb}.
 The quantum relative entropy belongs to a wider family of divergences between positive definite matrices. For example, Lieb's concavity theorem \cite{lieb} establishes the joint concavity of the functions
\begin{equation}
\label{eq:petzdivergence}
Q_{\alpha}(X|Y) = \tr(X^{\alpha} Y^{1-\alpha}) \qquad (\alpha \in [0,1])
\end{equation}
which are used to define the R\'enyi divergences $D_{\alpha}(X|Y) = \frac{1}{\alpha-1} \log Q_{\alpha}(X|Y)$, which converge to $D(X|Y)$ when $\alpha\to1$.

\paragraph{Noncommutative perspectives} The divergences above are strongly related to the notion of \emph{perspective} of a function. Given a function $g:(0,\infty)\to \RR$, its perspective is $P_g(x,y) = xg(y/x)$ defined for $x,y > 0$. It is well-known that if $g$ is concave, then $P_g$ is jointly concave in $(x,y)$. Note that when $g(x) = \log x$, we get $P_g(x,y) = -x\log(x/y)$ is the negative of the (scalar) relative entropy, and when $g(x) = x^{\beta}$, then $P_g(x,y) = x^{1-\beta} y^{\beta}$ is a geometric mean. In this paper, we are concerned with functions that arise from the generalization of the perspective to matrix arguments. The \emph{noncommutative perspective} of $g$ is defined by
\begin{equation}
\label{eq:perspective}
P_g(X,Y) = X^{1/2} g\left(X^{-1/2} Y X^{-1/2}\right) X^{1/2}
\end{equation}
for any $X,Y$ positive definite matrices. We recall that if $X$ is a Hermitian matrix with spectral decomposition $X = \sum_i \lambda_i v_i v_i^*$, then $g(X) = \sum_i g(\lambda_i) v_i v_i^*$. The function $g$ is operator concave if it satisfies Jensen's inequality in the positive semidefinite order $\psd$ (see~\eqref{eq:op-concave-def} for a precise definition).  When $g$ is operator concave, it can be shown that its perspective $P_g$ is jointly concave in $(X,Y)$, a fact that was established by Effros and generalized by Ebadian et al. \cite{effros2009matrix,ebadian2011perspectives}. The perspective of the logarithm function $g(x) = \log x$ (which is operator concave) corresponds to
\begin{equation}
\label{eq:Plog}
P_{\log}(X,Y) = -X^{1/2} \log(X^{1/2} Y^{-1} X^{1/2}) X^{1/2} \ ,
\end{equation}
which can be interpreted as the negative of an operator-valued relative entropy \cite{fujii1989relative}. Even though \eqref{eq:Plog} is distinct from \eqref{eq:qre} it can be shown \cite{effros2009matrix} that
\begin{equation}
\label{eq:qretensor}
D(X|Y) = -\p(P_{\log}(X\otimes I , I\otimes \bar Y))
\end{equation}
where $\p$ is an appropriate positive linear map. This shows that joint convexity of $D$ follows from joint concavity of $P_{\log}$.

\paragraph{Self-concordant barriers} The theory of self-concordant functions developed by Nesterov and Nemirovski \cite{nesterovnemirovski} plays a prominent role in convex optimization, and more particularly for interior-point methods. Consider a generic conic optimization problem
\begin{equation}
\label{eq:cp}
\min_x \quad \<c,x\> \quad \text{s.t.} \quad Ax=b, \; x \in K
\end{equation}
where $K \subset \RR^n$ is a closed convex cone and $A$ is a linear map. A \emph{self-concordant barrier} for $K$ is a convex function $F$ defined on the interior of $K$, such that $F(x) \to +\infty$ as $x$ approaches the boundary of $K$, the third derivative is bounded in terms of the second derivative, and the gradient is bounded in the norm defined by the (inverse) Hessian. Equipped with such a function, Nesterov and Nemirovski \cite{nesterovnemirovski} show that 
one can obtain an $\epsilon$-approximation to the optimal value of \eqref{eq:cp} in $O(\sqrt{\nu}\,\log(1/\epsilon))$ iterations of a path-following scheme, where $\nu$ is a parameter of the barrier function $F$. This scheme approximately follows the path $x^*(t) = \argmin \{ t \<c,x\> + F(x) : Ax=b \} \to x^*$ as $t\to \infty$ by applying Newton's method. Self-concordant barriers are known for various common convex cones such as the nonnegative orthant $\RR^n_+$, the second-order cone, the positive semidefinite cone $\S^n_+$, the exponential cone, and others, 
see e.g., \cite[Section 5.4.6]{nesterov2018lectures}. 
Implementations of interior-point algorithms have focused however on a small number of basic cones (essentially the nonnegative orthant, Cartesian products of second-order cones, and the positive semidefinite cone) because of their many symmetries that can be exploited algorithmically \cite{nesterovtodd}, and because most convex constraints that arise in practice can be formulated using these basic cones.
There have been however significant recent advances in developing practical interior-point solvers able to deal with generic convex sets $K$ via calls to a self-concordant barrier of $K$ \cite{karimi2020primal,coey2022performance,papp2022alfonso}.

\paragraph{Quantum relative entropy optimization} Optimization problems involving quantum entropies have attracted interest recently, and several approaches have been developed to deal with these problems. 
 The paper \cite{logapprox} showed how one can approximate, to high accuracy, optimization problems involving the quantum relative entropy using semidefinite programming. These approximations allow quantum relative entropy optimization problems to be solved using state-of-the-art algorithms for semidefinite programming. However, because these formulations usually require some form of ``lifting'', the resulting semidefinite program can be much larger than the original problem size; in fact, the semidefinite formulation of the quantum relative entropy function from \cite{logapprox} involves linear matrix inequalities of size $n^2 \times n^2$, resulting from the tensor product formulation in \eqref{eq:qretensor}.\footnote{We note however, that the semidefinite formulation of \eqref{eq:Plog} involves linear matrix inequalities of size $2n\times 2n$ only.} 

First-order methods have also been proposed to deal with specific optimization problems involving quantum entropies, such as the computation of quantum capacities, or various notions of quantum entanglement, see e.g., \cite{zinchenko2010numerical,sutter2015efficient,winick2018reliable,quantumblahutarimoto,you2021minimizing}. Recently, a second-order Gauss-Newton method was proposed for the problem of computing the rate of a quantum key distribution protocol (a particular instance of a quantum relative entropy optimization problem) \cite{hu2021robust}, however the method does not come with a quantitative convergence guarantee.

In \cite{faybusovich2020self}, the authors obtained self-concordant barriers for sets of the form $\{(X,Z) : g(X) \psd Z\}$ for any operator monotone function $g$.
 This implies, in particular, a self-concordant barrier for the epigraph of the relative entropy function $D(X|Y)$, when one of the arguments is fixed. Our results in this paper extend those of~\cite{faybusovich2020self} to all operator concave functions and, more importantly, to their noncommutative perspectives. This ultimately allows us to deal with various relative entropies and divergences.

\paragraph{Contributions} In this paper we give self-concordant barriers for convex sets defined in terms of the quantum relative entropy function, and related functions. Our results  show that one can solve quantum relative entropy optimization, without having to incur the lifting cost associated to semidefinite approximations.
In fact our results open the door to convergence guarantees for interior-point methods for quantum relative entropy optimization.  We note that some of the barriers we consider here have previously been conjectured to be self-concordant in \cite{karimi2019domain}.

Our first main theorem gives a self-concordant barrier for the matrix hypograph of 
the noncommutative perspective (defined in~\eqref{eq:perspective}) of any operator concave function $g:(0,\infty)\to \RR$. 
Here, and throughout, we denote by $\S_{++}^n$ the cone of $n\times n$ Hermitian positive definite matrices, 
and say that a linear map $\phi:\S^n\rightarrow \S^m$ is \emph{positive} if $\phi(X) \psd 0$ whenever $X \psd 0$. 
\begin{theorem}
\label{thm:main1}
Let $g:(0,\infty)\to \RR$ be an operator concave function and let $P_g$ be its noncommutative perspective. Let $\phi:\S^n\to\S^m$ be a positive linear map.
Then the function
\begin{equation}
\label{eq:scPg}
	(X,Y,Z) \mapsto -\log \det(\phi(P_g(X,Y)) - Z) - \log\det X - \log \det Y
\end{equation}
	defined on $\S^{n}_{++} \times \S^n_{++} \times \S^m$ is a $(2n+m)$-logarithmically homogeneous 
	self-concordant barrier for the closed convex cone
\begin{equation}
\label{eq:clhypo}
	\cl \{(X,Y,Z) \in \S^n_{++} \times \S^n_{++} \times \S^m : \phi(P_g(X,Y)) \psd Z\}.
\end{equation}
Moreover, this barrier is optimal in the sense that any self-concordant barrier for this cone has parameter at least $2n+m$.
\end{theorem}
\begin{remark}
We remark that the closure of the matrix hypograph \eqref{eq:clhypo} can be computed explicitly depending on the properties of the function $g$. See Remark \ref{rem:domainpersp} and Appendix \ref{sec:domainpersp}.
\end{remark}
In the following corollaries, we specialize the result above to particular classes of functions of interest. 
For two positive semidefinite matrices $X,Y$, we use the following notation which will be useful for the rest of the paper
\[
X \ll Y \iff \ker(Y) \subset \ker(X).
\]

First we consider the \emph{Belavkin-Staszewski relative entropy} function \cite{bsentropy}
\begin{equation}
\label{eq:DBS}
D_{BS}(X|Y) = \tr(X \log(X^{1/2} Y^{-1} X^{1/2}))
\end{equation}
which is jointly convex with domain $\{(X,Y) \in \S^n_+ \times \S^n_+ :X \ll Y \}$. We note that $D_{BS}(X|Y)$ and $D(X|Y)$ are distinct functions, even though they coincide when $X$ and $Y$ commute.
\begin{corollary}
	\label{cor:DBS}
Consider the epigraph of the Belavkin-Staszewski relative entropy function
\[
\epi(D_{BS}) = \{(X,Y,z) \in \S^n_+ \times \S^n_+ \times \RR : X \ll Y \text{ and } D_{BS}(X|Y) \leq z\} \ ,
\]
which is a closed convex set.
Then the function $(X,Y,z) \mapsto -\log(z - D_{BS}(X|Y)) - \log \det X - \log \det Y$ defined on $\S^n_{++} \times \S^n_{++} \times \RR$ is a $(2n+1)$-logarithmically homogeneous self-concordant barrier for $\epi(D_{BS})$.
	Moreover, this barrier is optimal in the sense that any self-concordant barrier for $\epi(D_{BS})$
	has parameter at least $2n+1$.
\end{corollary}
The next corollary deals with the  functions
\[
\hat{Q}_{\alpha}(X|Y) = \tr(X (X^{-1/2} Y X^{-1/2})^{1-\alpha}),
\]
which are concave for $\alpha \in [0,1]$ and convex for $\alpha \in [-1,0]\cup [1,2]$, and are used to define the geometric or maximal R\'enyi divergences \cite{matsumoto,fang2021geometric}.
\begin{corollary}
	\label{cor:R}{}
\begin{itemize}
\item For $\alpha \in [0,1]$, $\hat{Q}_{\alpha}$ is well-defined and concave on $\S^n_+ \times \S^n_+$. The function $(X,Y,z) \mapsto -\log(\hat{Q}_{\alpha}(X|Y) - z) - \log \det X - \log \det Y$ defined on $\S^n_{++} \times \S^n_{++} \times \RR$ is a $(2n+1)$-logarithmically homogeneous self-concordant barrier for
\[
\hypo(\hat{Q}_{\alpha}) = \{(X,Y,z) \in \S^n_+ \times \S^n_+ \times \RR : \hat{Q}_{\alpha}(X|Y) \geq z\}.
\]
		Moreover, this barrier is optimal in the sense that any self-concordant barrier for $\hypo(\hat{Q}_{\alpha})$
	has parameter at least $2n+1$.
\item For $\alpha \in [-1,0) \cup (1,2]$, $\hat{Q}_{\alpha}$ is well-defined and convex on
\[
\hat{\cD}_{\alpha} = \begin{cases}
	\{(X,Y) \in \S^n_+ \times \S^n_+ : Y \ll X\} & \textup{for $\alpha \in [-1,0)$}\\
	\{(X,Y) \in \S^n_+ \times \S^n_+ : X \ll Y\} & \textup{for $\alpha \in (1,2]$}.
\end{cases}
\]
The function $(X,Y,z) \mapsto -\log(z - \hat{Q}_{\alpha}(X|Y)) - \log \det X - \log \det Y$ defined on $\S^n_{++} \times \S^n_{++} \times \RR$ is a $(2n+1)$-self-concordant barrier for
\[
\epi(\hat{Q}_{\alpha}) = \{(X,Y,z) \in \S^n_+ \times \S^n_+ \times \RR : (X,Y) \in \hat{\cD}_{\alpha} \text{ and } \hat{Q}_{\alpha}(X|Y) \leq z\}.
\]
		Moreover, this barrier is optimal in the sense that any self-concordant barrier for $\epi(\hat{Q}_{\alpha})$
	has parameter at least $2n+1$.
\end{itemize}
\end{corollary}

The next theorem allows us to obtain self-concordant barriers for the quantum relative entropy \eqref{eq:qre}, and more generally for functions of the same form as \eqref{eq:qretensor}.
\begin{theorem}
\label{thm:main2}
	Let $g:(0,\infty)\to \RR$ be an operator concave function, and let $P_g$ as in \eqref{eq:perspective} be its noncommutative perspective. Let $\phi:\S^{n_1n_2} \to \S^m$ be a positive linear map and consider the (concave) function $Q_g^\phi:\S_{++}^{n_1}\times \S_{++}^{n_2}\rightarrow \S^m$ defined by
\begin{equation}
\label{eq:Qgdef}
	Q_g^{\phi}(X|Y) = \phi(P_g(X\otimes I, I \otimes \bar Y)).
\end{equation}
Then the function
\begin{equation}
\label{eq:scQg}
(X,Y,Z) \mapsto -\log\det(Q_g^\phi(X|Y)-Z) -\log \det X - \log \det Y
\end{equation}
	is a $(n_1+n_2+m)$-logarithmically homogeneous self-concordant barrier for 
\[
	\cl\hypo(Q_g^\phi) = \cl \left\{(X,Y,Z) \in \S^{n_1}_{++}\times \S^{n_2}_{++} \times \S^m : Q_g^\phi(X|Y) \psd Z \right\}.
\]
	Moreover,
	this barrier is optimal in the sense that any self-concordant barrier for $\cl\hypo(Q_{g}^\phi)$
	has parameter at least $n_1+n_2+m$.
\end{theorem}
\begin{remark}
	In Theorem~\ref{thm:main2}, the notation $\bar{Y}$ denotes the entry-wise complex conjugate of the 
	Hermitian matrix $Y$. We have stated \eqref{eq:Qgdef} in this way because the  
	positive linear map $\p$ that sends $X\otimes \bar{Y}$ to $\tr(XY)$ plays a particularly important role 
	in the applications of this result that follow. 
	(In the paragraph related to Kronecker products in Section~\ref{sec:prelim} we discuss this positive linear map further, and 
	will see why the entry-wise complex conjugate appears.)
\end{remark}
We obtain the quantum relative entropy, and the divergences \eqref{eq:petzdivergence} as a special case when $g$ is respectively the logarithm function, and the  power functions.
\begin{corollary}
	\label{cor:qre}
Let $D(X|Y) = \tr(X\log X - X \log Y)$ be the quantum relative entropy function defined for positive semidefinite matrices $X,Y \in \S^n_+$ such that $X \ll Y$, and consider its epigraph
\[
\epi(D) = \{(X,Y,z) \in \S^n_+ \times \S^n_+ \times \RR : X \ll Y \text{ and } D(X|Y) \leq z\}.
\]
Then the function $(X,Y,z) \mapsto -\log(z - D(X|Y)) - \log \det X - \log \det Y$ defined on $\S^n_{++} \times \S^n_{++} \times \RR$ is a $(2n+1)$-logarithmically homogeneous self-concordant barrier for $\epi(D)$.
		Moreover, this barrier is optimal in the sense that any self-concordant barrier for $\epi(D)$
	has parameter at least $2n+1$.
\end{corollary}
The next corollary deals with the  functions
$Q_{\alpha}(X|Y) = \tr(X^{\alpha} Y^{1-\alpha})$ which are concave for $\alpha \in [0,1]$ and convex for $\alpha \in [-1,0]\cup [1,2]$, and are used to define the Petz R\'enyi divergences \cite{petz1986quasi} (see also \cite[Section 4.4]{tomamichel2015quantum}).
\begin{corollary}
	\label{cor:Qalpha}
\begin{itemize}
\item For $\alpha \in [0,1]$, $Q_{\alpha}$ is well-defined and concave on $\S^n_+ \times \S^n_+$. The function $(X,Y,z) \mapsto -\log(Q_{\alpha}(X|Y) - z) - \log \det X - \log \det Y$ defined on $\S^n_{++} \times \S^n_{++} \times \RR$ is a $(2n+1)$-logarithmically homogeneous self-concordant barrier for
\[
\hypo(Q_{\alpha}) = \{(X,Y,z) \in \S^n_+ \times \S^n_+ \times \RR : Q_{\alpha}(X|Y) \geq z\}.
\]
		Moreover, this barrier is optimal in the sense that any self-concordant barrier for $\hypo(Q_{\alpha})$
	has parameter at least $2n+1$.
\item For $\alpha \in [-1,0) \cup (1,2]$, $Q_{\alpha}$ is well-defined and convex on
\[
\cD_{\alpha} = \begin{cases}
	\{(X,Y) \in \S^n_+ \times \S^n_+ : Y \ll X\} & \textup{for $\alpha \in [-1,0)$}\\
	\{(X,Y) \in \S^n_+ \times \S^n_+ : X \ll Y\} & \textup{for $\alpha \in (1,2]$}.
\end{cases}
\]
The function $(X,Y,z) \mapsto -\log(z - Q_{\alpha}(X|Y)) - \log \det X - \log \det Y$ defined on $\S^n_{++} \times \S^n_{++} \times \RR$ is a $(2n+1)$-logarithmically homogeneous self-concordant barrier for
\[
\epi(Q_{\alpha}) = \{(X,Y,z) \in \S^n_+ \times \S^n_+ \times \RR : (X,Y) \in \cD_{\alpha} \text{ and } Q_{\alpha}(X|Y) \leq z\}.
\]
		Moreover, this barrier is optimal in the sense that any self-concordant barrier for $\epi(Q_{\alpha})$
	has parameter at least $2n+1$.
\end{itemize}
\end{corollary}


\section{Preliminaries}
\label{sec:prelim}

In this section we summarize necessary preliminary background and notation related to self-concordant functions and barriers, matrix 
monotone and matrix concave functions, and the noncommutative perspective operation. 

\paragraph{Directional derivatives} If $E$ and $F$ are two real vector spaces, and $f:\dom f \subset E\to F$ is a $C^k$ function defined on an open subset $\dom f \subset E$, we define
\[
D^k f(x)[h] = \left.\frac{d^k}{dt^k} f(x+th)\right|_{t=0}
\]
for $x \in \dom f$ and $h\in E$. Note that $D^k f(x)[h]$ is homogeneous of degree $k$ in $h$, i.e., $D^1 f(x)[h]$ is linear, $D^2 f(x)[h]$ is quadratic, etc. The Taylor expansion tells us that
\[
f(x+h) = f(x) + D^1 f(x)[h] + \frac{1}{2} D^2 f(x)[h] + \dots + \frac{1}{k!} D^k f(x)[h] + O(h^{k+1}).
\]

\paragraph{Self-concordant functions and barrier} A convex function $f:\RR^n\to \RR \cup \{+\infty\}$ is \emph{self-concordant} \cite[Def. 5.1.1]{nesterov2018lectures} if the following three conditions hold:
\begin{itemize}
\item $\dom(f) = \{x \in \RR^n : f(x) < +\infty\}$ is open
\item $\epi(f) = \{(x,t) \in \RR^n \times \RR : f(x) \leq t\}$ is closed
\item $f$ is $C^3$ on its domain and
\begin{equation}
\label{eq:scdef}
|D^3 f(x)[h]| \leq 2 (D^2 f(x)[h])^{3/2}
\end{equation}
for all $x \in \dom(f)$ and all $h \in \RR^n$.
\end{itemize}
We say that it is a \emph{barrier} for the closed convex set $Q \subset \RR^n$ if $\cl(\dom f) = Q$; moreover, we say that it is a $\nu$-barrier \cite[Def. 5.3.2]{nesterov2018lectures} if
\begin{equation}
	\label{eq:nu-def}
2 Df(x)[h] - D^2 f(x)[h] \leq \nu
\end{equation}
for all $x \in \dom(f)$ and $h \in \RR^n$. (If $\nabla^2 f(x)$ is invertible for all $x \in \dom(f)$, then condition~\eqref{eq:nu-def} is equivalent to $\langle \nabla f(x),[\nabla^2f(x)]^{-1}\nabla f(x)\rangle \leq \nu$ for all $x \in \dom(f)$, i.e., the gradient is bounded by $\nu$ in the quadratic form defined by the inverse Hessian~\cite[Eq.~(5.3.5)]{nesterov2018lectures}. This is obtained by explicitly maximizing the left hand side of~\eqref{eq:nu-def} with respect to $h$.)
When $\dom(f)$ is a convex cone, we say that $f$ is $\nu$-\emph{logarithmically homogeneous} \cite[Def. 5.4.1]{nesterov2018lectures} if
\[
f(\tau x) = f(x) - \nu \log \tau\quad \forall x \in \dom(f) \text{ and } \tau > 0.
\]
If $f$ is $\nu$-logarithmic homogeneous then it is automatically a $\nu$-barrier \cite[Lem. 5.4.3]{nesterov2018lectures}. Note that if $\cA:\RR^m\to \RR^n$ is a linear map, and $f$ is self-concordant, then so is $f\circ \cA:\RR^m\to \RR$ with domain $\cA^{-1}(\dom f)$. Furthermore, if $f$ is a $\nu$-barrier for $\cl \dom f$, then $f \circ \cA$ is a $\nu$-barrier for $\cl(\cA^{-1}(\dom f))$  \cite[Thm 5.3.3]{nesterov2018lectures}.

\begin{example}
\label{ex:-logdet}
Consider the function $f(X) = -\log \det X$ defined on $\S^n_{++}$. For this $f$, if
	we let $A = X^{-1/2} H X^{-1/2}$ then we have
\[
\begin{aligned}
Df(X)[H] &= -\tr(X^{-1} H) = -\tr(A)\\
D^2 f(X)[H] &= \tr(X^{-1} H X^{-1} H) = \|A\|_F^2\\
D^3 f(X)[H] &= -2\tr(X^{-1} H X^{-1} H X^{-1} H) = -2\tr(A^3).
\end{aligned}
\]
Self-concordance follows from the fact that $|\tr(A^3)|^{1/3} \leq |\tr(A^2)|^{1/2}$. Also we see that $f$ is $n$-logarithmically homogeneous.
\end{example}

\paragraph{Matrix monotone and matrix concave functions} If $g:I\to \RR$ where $I\subset \RR$ is an interval, and $X$ is a Hermitian matrix with spectral decomposition $X = \sum_{i} \lambda_i v_i v_i^*$ with eigenvalues $\lambda_i \in I$ for all $i$, we let $g(X) = \sum_{i} g(\lambda_i) v_i v_i^*$. We say that $g$ is \emph{operator monotone} if
\[
X \psd Y \implies g(X) \psd g(Y).
\]
We say that $g:I\to \RR$ is \emph{operator concave} if 
for all Hermitian matrices $X,Y$ (of any size) with eigenvalues in $I$, and all $\lambda \in [0,1]$,
\begin{equation}
	\label{eq:op-concave-def}
g(\lambda X + (1-\lambda) Y) \psd \lambda g(X) + (1-\lambda) g(Y).
\end{equation}
In this paper we are going to be primarily concerned with functions $g:(0,\infty)\to \RR$
that are operator concave. 
Important examples of operator concave functions
are $g(x) = \log x$, $g(x) = x^{\alpha}$ for $\alpha \in [0,1]$ and $g(x) =
-x^{\alpha}$ for $\alpha \in [-1,0]$. These functions happen to be operator
monotone too. However not all operator concave functions are operator monotone,
for example the functions $g(x)=-x^{\alpha}$ for $\alpha \in (1,2]$ are
operator concave but not operator monotone. Note, also, that monotone functions need not 
be operator monotone, and that concave functions need not be operator concave. For example
$-e^{-x}$ is monotone and concave, but is neither operator monotone nor operator concave on any interval.
Although we mostly focus on examples 
related to the logarithm and power functions, the class of operator concave functions
includes other interesting examples, such as the negative of the 
log gamma function~\cite{uchiyama2010operator}.

The following theorem (see Appendix~\ref{sec:intrepoperatorconcave}) shows that any operator concave function $g:(0,\infty)\to \RR$ can be expressed as an integral of rational functions.
\begin{theorem}
\label{thm:intrep}
If $g:(0,\infty) \to \RR$ is an operator concave function, then $g$ is analytic and there is a finite positive measure $\mu$ supported on $[0,1]$ such that
\begin{equation}
\label{eq:intrep}
g(x) = g(1) + g'(1)(x-1) - \int_{0}^{1} \frac{(x-1)^2}{1+s(x-1)} d\mu (s) \qquad \forall x > 0.
\end{equation}
\end{theorem}
The integrand can be shown to be operator concave for any $s \in [0,1]$ using the Schur complement lemma, since
\begin{equation}
	-(X-I) (I+s(X-I))^{-1}(X-I) \psd T \iff \begin{bmatrix} -T & X-I\\ X-I & I+s(X-I) \end{bmatrix} \psd 0.
\end{equation}
The theorem above tells us that any operator concave function is essentially an (infinite) conic combination of such functions, together with affine functions of $x$.

\paragraph{Noncommutative perspective} If $g:(0,\infty) \to \RR$, we define the noncommutative perspective of $g$ by
\begin{equation}
\label{eq:perspective2}
P_g(X,Y) = X^{1/2} g\left(X^{-1/2} Y X^{-1/2}\right) X^{1/2}
\end{equation}
for any $X,Y$ Hermitian positive definite matrices. This generalizes the perspective for scalar arguments $P_g(x,y) = xg(y/x)$. If $g$ is operator concave, then it has been shown \cite{effros2009matrix,ebadian2011perspectives} that $P_g$ is jointly operator concave in $(X,Y)$, in the sense that
\begin{equation}
\label{eq:supadd1}
P_g( \lambda_1 X_1 + \lambda_2 X_2, \lambda_1Y_1 + \lambda_2Y_2 ) \psd \lambda_1 P_g(X_1,Y_1) + \lambda_2 P_g(X_2,Y_2)
\end{equation}
for all $X_1,Y_1,X_2,Y_2$ positive definite, and $\lambda_1,\lambda_2 \geq 0$ such that $\lambda_1 + \lambda_2 = 1$. Since $P_g$ is $1$-homogeneous, i.e., $P_g(\lambda X, \lambda Y) = \lambda P_g(X,Y)$ for all $\lambda \geq 0$, concavity of $P_g$ is actually equivalent to sup-additivity:
\begin{equation}
\label{eq:supadd2}
P_g( X_1 +  X_2, Y_1 + Y_2 ) \psd  P_g(X_1,Y_1) +  P_g(X_2,Y_2).
\end{equation}
The \emph{transpose} of $g$ is defined by $\hat{g}(x) = xg(1/x)$. If $g$ is operator concave, then $\hat{g}$ is operator concave too, and we have, for all $X, Y \pd 0$
\begin{equation}
\label{eq:transposeid}
P_g(X,Y) = P_{\hat g}(Y,X),
\end{equation}
see e.g., \cite[Lemma 2.1]{hiai2017different}.

\begin{remark}[Domain of the matrix perspective]
\label{rem:domainpersp}
So far, we have restricted the domain of $P_g$ to pairs of positive definite matrices $(X,Y)$. Depending on the function $g$, the domain can be extended to ensure closedness of the hypograph of $P_g$. This is treated in detail in Appendix \ref{sec:domainpersp}. In summary, four cases need to be considered, depending on whether $g(0^+):=\lim_{x\to0} g(x)$ and $\hat{g}(0^+)=\lim_{x\to0} \hat{g}(x)$ are finite or infinite:
\begin{itemize}
\item If both $g(0^+)$ and $\hat{g}(0^+)$ are finite (such as $g(x) = x^{\alpha}$ for $\alpha \in [0,1]$), then $P_g$ can be extended to all pairs of positive semidefinite matrices.
\item If $g(0^+) = -\infty$ and $\hat{g}(0^+)$ is finite (such as $g(x) = \log x$ or $g(x) = -x^{\alpha}$ for $\alpha \in [-1,0)$), then $P_g$ can be extended to all pairs $(X,Y)$ such that $X \ll Y$, i.e., $\ker(Y) \subset \ker(X)$.
\item If $g(0^+)$ is finite and $\hat{g}(0^+) = -\infty$ (such as $g(x) = -x^{\alpha}$ for $\alpha \in (1,2]$), then $P_g$ can be extended to all pairs $(X,Y)$ such that $Y \ll X$.
\item Finally if both $g(0^+)$ and $\hat g(0^+)$ are infinite (such as $g(x)=-x^\alpha - x^{1-\alpha}$ for $\alpha\in (1,2]$), then $P_g$ can be extended to all pairs $(X,Y)$ such that $\ker(X) = \ker(Y)$.
\end{itemize}
\end{remark}

\paragraph{Kronecker products} The Kronecker product of two  matrices $X \in \CC^{m_1\times n_1}$ and $Y \in \CC^{m_2 \times n_2}$ is denoted $X\otimes Y \in \CC^{m_1 m_2 \times n_1 n_2}$ and defined by
\[
(X \otimes Y)_{i_1 i_2, j_1j_2} = X_{i_1 j_1} Y_{i_2 j_2}, \quad 1\leq i_1 \leq m_1, \; 1\leq i_2 \leq m_2,  \; 1\leq j_1 \leq n_1,  \; 1\leq j_2 \leq n_2.
\]
If $X$ and $Y$ are Hermitian, with eigenvalue decompositions $X = U_1 \Lambda_1 U_1^*$ and $Y=U_2 \Lambda_2 U_2^*$, where $U_1,U_2$ are unitaries and $\Lambda_1,\Lambda_2$ diagonal, then $X\otimes Y$ is Hermitian with eigenvalue decomposition
\[
X\otimes Y = (U_1 \otimes U_2) (\Lambda_1 \otimes \Lambda_2) (U_1 \otimes U_2)^*.
\]
As such if $X,Y \psd 0$, then $X\otimes Y \psd 0$, and if $X,Y \pd 0$ then $X\otimes Y \pd 0$. In this case we get for any $\alpha \in \RR$, $(X\otimes Y)^{\alpha} = X^{\alpha} \otimes Y^{\alpha}$ and $\log(X\otimes Y) = \log(X) \otimes I + I\otimes \log(Y)$. 

If we let $\e \in \CC^{n^2}$ be defined by $\e_{i_1i_2} = 1$ if $i_1=i_2$ and $0$ otherwise (i.e., $\e$ is obtained from the $n\times n$ identity matrix by stacking the columns in an $n^2$-vector), then for any $X,Y \in \CC^{n\times n}$ we have
\begin{equation}
\label{eq:eXYe}
\e^*(X\otimes Y) \e = \sum_{1\leq i,j \leq n} X_{ij} Y_{ij}.
\end{equation}
We define the linear map $\p:\H^{n^2}\to \RR$ by
\begin{equation}
\label{eq:pZ}
\p(Z) = \psi^* Z \psi
\end{equation}
and note the following two important properties which will be useful for rest of the paper: $\p(Z) \geq 0$ for any $Z \in \H^{n^2}_+$ and for any $X,Y \in \H^n$, we have
\begin{equation}
\label{eq:propp}
	\p(X\otimes \bar Y) = \sum_{1\leq i,j\leq n} X_{ij}\bar{Y}_{ij} = \sum_{1\leq i,j\leq n} X_{ij} Y_{ji} = \tr(XY).
\end{equation}


\section{Proofs}
\label{sec:pfs}

In this section we prove Theorems \ref{thm:main1} and~\ref{thm:main2}, as well as Corollaries~\ref{cor:DBS},~\ref{cor:R},~\ref{cor:qre}, and~\ref{cor:Qalpha}. The strategy of the argument is summarized in Section~\ref{sec:nesterov} and the key technical conditions that need to be checked for the construction of the barriers are in Section~\ref{sec:compat}. The proofs of Theorems~\ref{thm:main1}
and~\ref{thm:main2} also appear in Section~\ref{sec:compat}, with the exception of the proofs of the lower bounds on the barrier parameters. These appear in Section~\ref{sec:lb}, which is devoted to establishing
tight lower bounds on the barrier parameters for the self-concordant barriers we construct.
Finally, in Section~\ref{sec:pfassembly}, we show how to specialize Theorems~\ref{thm:main1} and~\ref{thm:main2} to establish
Corollaries~\ref{cor:DBS},~\ref{cor:R},~\ref{cor:qre}, and~\ref{cor:Qalpha}.

\subsection{The compatibility condition of Nesterov and Nemirovski}
\label{sec:nesterov}
Our arguments make crucial use of  general methods, due to Nesterov and Nemirovski~\cite[Proposition 5.1.7]{nesterovnemirovski}, to construct self-concordant barriers for
the hypographs of functions that satisfy the following generalized concavity property.
\begin{definition}
	Let $E$ be a finite-dimensional real vector space. A function $\xi:\dom\xi \subset E\to \S^m$ is \emph{$\S^m_+$-concave} if
\[
	\xi(\lambda x + (1-\lambda)y) \psd \lambda \xi(x) + (1-\lambda) \xi(y)
\]
	for all $\lambda\in [0,1]$ and $x,y\in \dom\xi$. 
\end{definition}
Note that if $\dom \xi$ is open, and $\xi:\dom\xi\subset E \to\S^m$ is $C^2$ on its domain, then $\xi$ is $\S_+^m$-concave if and only if $D^2\xi(x)[h] \nsd 0$ for all $x\in \dom \xi$ and all $h\in E$. 
The \emph{hypograph} of an $\S_+^m$-concave
function $\xi$ is the convex set 
\begin{equation}
\label{eq:hypo}
\hypo(\xi) = \left\{ (x,z) \in \dom \xi \times \S^m : \xi(x) \psd z \right\}.
\end{equation}
The following  compatibility condition due to Nesterov and Nemirovski provides a way to  establish self-concordance for a natural barrier function associated to the hypograph of a $\S_+^m$-concave function $\xi$.
\begin{definition}
Let $E$ be a finite-dimensional real vector space, and let $\xi:\dom \xi \subset E \to \S^m$ 
	be a $C^3$, $\S^m_+$-concave function defined on an open domain $\dom \xi \subset E$.
	Let $F$ be a self-concordant barrier for $\cl \dom \xi$. We say that $\xi$ is \emph{$\beta$-compatible with $F$}, if there is a constant $\beta \geq 1$ such that 
\begin{equation}
\label{eq:compat}
D^3 \xi(x)[h] \nsd 3 \beta (D^2 F(x)[h])^{1/2} (-D^2 \xi(x)[h])
\end{equation}
for all $x \in \dom(\xi)$ and all $h\in E$.
\end{definition}
Crucial to our arguments is the following special case of~\cite[Theorem 5.4.4]{nesterov2018lectures} (closely related to~\cite[Proposition 5.1.7]{nesterovnemirovski}).
\begin{theorem}[{\cite[Theorem 5.4.4]{nesterov2018lectures}}]
	\label{thm:nesterov}
Let $E$ be a finite-dimensional real vector space, and let $\xi:\dom \xi \subset E \to \S^m$ be a $C^3$ 
	$\S^m_+$-concave function defined on an open domain $\dom \xi \subset E$.
If $F$ is a self-concordant barrier for $\cl \dom \xi$ with parameter $\nu$, and $\xi$ is $\beta$-compatible with $F$,
then $(x,z)\mapsto -\log \det(\xi(x) - z) + \beta^3 F(x)$ is a self-concordant barrier for $\cl \hypo(\xi)$ with parameter $m+\beta^3 \nu$.
\end{theorem}
\begin{proof}
In the notations of~\cite[Theorem 5.4.4]{nesterov2018lectures}, we take $E_1 = E$, $E_2 = E_3 = \H^m$, $K = \H^m_+ \subset E_2$, $Q_2 = \{(y,z) \in \H^m \times \H^m : y \psd z\}$ and $\Phi(y,z) = -\log\det(y-z)$, a $m$-self-concordant barrier for $Q_2$. It is clear that any element of  $K\times \{0\}$ is a recession direction for $Q_2$. Finally we note that
\[
\cl \hypo(\xi) = \cl \{(x,z) \in \dom \xi \times \H^m : \exists y, \; \xi(x) \psd y, \; (y,z) \in Q_2\}.
\]
\end{proof}
The usefulness of the theorem above lies in the remarkable properties of the compatibility condition \eqref{eq:compat}.
First, note that \eqref{eq:compat} is \emph{linear} in $\xi$. Therefore, if $(\xi_s)$ is a family of functions defined on the same domain which are all $\beta$-compatible with $F$ (a self-concordant barrier for the closure of their domain), then for any appropriate positive measure $\mu$, the function $\xi(x) = \int \xi_s(x) d\mu(s)$ is also $\beta$-compatible with $F$, as long as we can exchange the order of integration and differentiation. The theorem then gives a self-concordant barrier for the hypograph of $\xi(x)$. This fact, together with the integral representation of operator concave functions (Theorem~\ref{thm:intrep}), allows us to focus
on establishing the compatibility condition for functions of the form
\begin{equation}
	\label{eq:xis}
\xi_s(X,Y) = -(Y-X)(X+s(Y-X))^{-1}(Y-X),
\end{equation}
where $s\geq 0$. Note that $\xi_s(X,Y)$ is exactly the perspective of $x\mapsto -\frac{(x-1)^2}{1+s(x-1)}$.

Second, it can be shown that the compatibility condition behaves well under composition with positive linear maps. This is shown in the next proposition, a special case of~\cite[Proposition 5.1.9]{nesterovnemirovski}.
\begin{proposition}
	\label{prop:compat-pos}
	Let $E$ be a finite-dimensional real vector space, and let $\xi:\dom \xi \subset E \to \S^m$ be a $C^3$ 
	$\S^m_+$-concave function defined on an open domain $\dom \xi \subset E$.
	Let $F$ be a self-concordant barrier for $\cl \dom \xi$ such that $\xi$ is $\beta$-compatible with $F$.  If $\phi:\S^m\rightarrow \S^k$ is a positive linear map then $\phi\circ \xi : \dom \xi \to \S^k$ 
	is $\S_+^k$-concave and is $\beta$-compatible with $F$.
\end{proposition}
\begin{proof}
	First, we note that $D^k(\phi \circ \xi)(x)[h] = \phi(D^k \xi(x)[h])$ for all $x\in \dom\xi$, all $h\in E$, and all $k\geq 0$. 
	To see that $\phi\circ \xi$ is $\S_+^k$-concave, 
	we observe that $D^2(\phi\circ \xi)(x)[h] = \phi(D^2\xi(x)[h]) \nsd 0$ for all $x\in \dom\xi$ and all $h\in E$, because $\phi$
	is positive and $\xi$ is $\S_+^m$-concave. To see that $\phi\circ \xi$ is $1$-compatible with $F$, we simply apply 
	$\phi$ to both sides of the compatibility condition for $\xi$~\eqref{eq:compat} to give the compatibility condition for $\phi\circ \xi$.
\end{proof}

We now show how the required compatibility condition for the main classes of functions we consider in this paper
follow from the compatibility condition for the functions $\xi_s$ for $s\geq 0$ defined in \eqref{eq:xis}.
\begin{proposition}
	\label{prop:conic}
	Let $F_{n_1,n_2}:\S^{n_1}_{++} \times \S^{n_2}_{++} \to \RR$, $F_{n_1,n_2}(X,Y) = -\log\det(X) - \log\det(Y)$, which is a self-concordant barrier for $\S_+^{n_1}\times \S_+^{n_2}$ for any $n_1,n_2 \geq 1$.
\begin{itemize}
	\item If $(X,Y) \in \S^{n}_{++} \times \S^{n}_{++} \mapsto \xi_s(X,Y)$ is $1$-compatible with $F_{n,n}$ for all $s \in [0,1]$ then $(X,Y)\mapsto \phi(P_g(X,Y))$ is $1$-compatible with $F_{n,n}$ for all operator concave functions $g:(0,\infty)\to\RR$ and all positive maps $\phi:\S^n\rightarrow \S^m$.
	\item If $(X,Y) \in \S^{n_1}_{++} \times \S^{n_2}_{++}\mapsto \xi_s(X\otimes I,I\otimes \bar Y)$ is $1$-compatible with $F_{n_1,n_2}$ for all $s \in [0,1]$ then $(X,Y)\mapsto \phi(P_g(X\otimes I,I\otimes \bar Y))$ is $1$-compatible with $F_{n_1,n_2}$ for all operator concave functions $g:(0,\infty)\to\RR$ and all positive
		maps $\phi:\S^{n_1 n_2}\rightarrow \S^m$.
\end{itemize}
\end{proposition}
\begin{proof}
	It is clear from the definition that the set of functions with 
	domain $\S_{++}^{n_1}\times \S_{++}^{n_2}$ that are $1$-compatible 
	with $F_{n_1,n_2}$ is closed under taking conic combinations and 
	that any linear function is $1$-compatible with $F_{n_1,n_2}$. 
	
	Since $g$ is operator concave, it has an integral representation as in Theorem~\ref{thm:intrep}. Then, for any $X,Y \pd 0$, 
	$P_g(X,Y)$ can be expressed as
	\begin{equation}
		\label{eq:intP}
	\begin{aligned}
		P_g(X,Y) &= g(1) X + g'(1)(Y-X) + \int_{0}^{1} \xi_s(X,Y) d\mu(s).
	\end{aligned}
	\end{equation}
	The first two terms in \eqref{eq:intP} are linear and hence are 1-compatible with $F_{n,n}$.
	The integrand is $1$-compatible with $F_{n,n}$ for all $s \in [0,1]$ by assumption, i.e.,
	\[
	D^3 \xi_s(X,Y)[H,V] \nsd 3 \beta (D^2 F_{n,n}(X,Y)[H,V])^{1/2} (-D^2 \xi_s(X,Y)[H,V])
	\]
	for all $X,Y \pd 0$, and $H,V \in \H^n$.
	By integrating this matrix inequality with respect to the measure $\mu$, and using Theorem \ref{thm:intrepdirectionalderiv} from Appendix \ref{sec:intrepoperatorconcave} which shows that the directional derivatives of $P_g$ are obtained by integrating the directional derivatives of the $\xi_s$, we get that
	\[
	D^3 P_g(X,Y)[H,V] \nsd 3 \beta (D^2 F_{n,n}(X,Y)[H,V])^{1/2} (-D^2 P_g(X,Y)[H,V]),
	\]
	i.e., that $P_g$ is 1-compatible with $F_{n,n}$.
	 The same is true for $\phi(P_g(X,Y))$ using Proposition~\ref{prop:compat-pos}, since $\phi$ is a positive linear map.

The same proof applies to show the 1-compatibility of $(X,Y) \mapsto \phi(P_g(X\otimes I, I \otimes \bar Y))$ with $F_{n_1,n_2}$, assuming the 1-compatibility of $(X,Y) \mapsto \xi_s(X\otimes I, I \otimes \bar Y)$ for all $s \in [0,1]$ (see Remark \ref{rem:intrepdirectionalderivtensor}).
\end{proof}

\subsection{Compatibility results and proofs of Theorems~\ref{thm:main1} and~\ref{thm:main2}}
\label{sec:compat}
We begin with a simple technical result. For $A\in \S^n$, let $\|A\|$ denote the spectral norm. 
\begin{lemma}
 	\label{lem:eta-bnd}
 	Let $X,Y\in \S_{++}^n$ and let $H,V\in \S^n$ be fixed matrices. Let $\eta = \max\{\|X^{-1/2}HX^{-1/2}\|,$ $\|Y^{-1/2}VY^{-1/2}\|\}$. Then 
 	\[\|((1-s)X+sY)^{-1/2}((1-s)H+sV)((1-s)X+sY)^{-1/2}\| \leq\eta\]
 	for all $s\in [0,1]$.
 \end{lemma}
 \begin{proof}
 	From the definition of $\eta$ we know that $-\eta I \nsd X^{-1/2}HX^{-1/2} \nsd \eta I$ and $-\eta I \nsd Y^{-1/2}VY^{-1/2} \nsd \eta I$. Therefore
 	$-\eta X \nsd H \nsd \eta X$ and $-\eta Y \nsd V \nsd \eta Y$. Since $s\geq 0$ and $1-s\geq 0$, taking 
 	the appropriate convex combination of these inequalities gives $-\eta((1-s)X+sY) \nsd (1-s)H+sV \nsd \eta((1-s)X+sY)$. The result follows by multiplying on both sides by $(sY+(1-s)X)^{-1/2}$, which exists because $sY+(1-s)X$ is positive definite.
 \end{proof}

Our main lemma for the proof of Theorem \ref{thm:main1} is the following.
\begin{lemma}
	\label{lem:compat1}
	For any $s \geq 0$, let $\xi_s(X,Y) = -(Y-X)(sY+(1-s)X)^{-1}(Y-X)$ be defined on the open set $\S^n_{++} \times \S^n_{++}$. 
	Then $\xi_s$ is 1-compatible with $F(X,Y) = -\log \det X-\log \det Y$.
\end{lemma}

\begin{proof}
	For $s\in [0,1]$ let $\Ybar = (1-s)X+sY$, $\Vbar = (1-s)H+sV$ and $\Bbar = \Ybar^{-1/2}\Vbar\Ybar^{-1/2}$. 
	The directional derivatives of $\xi_s$ are given in Lemma~\ref{lem:diff-facts}.
	For $k\geq 2$ we obtain
		\[	D^k\xi_s(X,Y)[H,V] 
		= -k{!}\begin{bmatrix} \Ybar^{-\frac{1}{2}}(Y-X) \\ \Ybar^{-\frac{1}{2}}(V-H)\end{bmatrix}^*
			\begin{bmatrix} (-\Bbar)^{k} & (-\Bbar)^{k-1}\\(-\Bbar)^{k-1} & (-\Bbar)^{k-2}\end{bmatrix}\begin{bmatrix}\Ybar^{-\frac{1}{2}}(Y-X)\\\Ybar^{-\frac{1}{2}}(V-H)\end{bmatrix}.
				\]
For $F(X,Y) = -\log \det X-\log \det Y$, recall from Example \ref{ex:-logdet} that 
\[
D^2 F(X,Y)[H,V] = \|X^{-1/2} H X^{-1/2}\|_F^2+\|Y^{-1/2} V Y^{-1/2}\|_F^2.
\]
To prove 1-compatibility of 
	$\xi_s$ with $F$, it suffices to prove the following inequality
\begin{equation}
\label{eq:compatxicmain}
\begin{bmatrix}
\Bbar^3 & -\Bbar^2\\ -\Bbar^2 & \Bbar
\end{bmatrix}
\nsd
 \sqrt{\|X^{-1/2} H X^{-1/2}\|_F^2+\|Y^{-1/2} V Y^{-1/2}\|_F^2} \begin{bmatrix}
\Bbar^2 & -\Bbar\\ -\Bbar & I
\end{bmatrix}.
\end{equation}
	This is because the $1$-compatibility condition for 
	$\xi_s$ follows from~\eqref{eq:compatxicmain} by 
	applying the positive linear map 
	\[ Z \mapsto \begin{bmatrix} \Ybar^{-1/2}(Y-X) \\ \Ybar^{-1/2}(V-H)\end{bmatrix}^*Z\begin{bmatrix} \Ybar^{-1/2}(Y-X) \\ \Ybar^{-1/2}(V-H)\end{bmatrix}\]
		to both sides.

	We now focus on establishing~\eqref{eq:compatxicmain}. 
	The matrix on the left-hand side of~\eqref{eq:compatxicmain} can be expressed as
\[
\begin{bmatrix}
\Bbar^3 & -\Bbar^2\\ -\Bbar^2 & \Bbar
\end{bmatrix}
= \begin{bmatrix} \Bbar \\ -I \end{bmatrix}
\Bbar
\begin{bmatrix} \Bbar & -I\end{bmatrix} 
	\nsd \lambda_{\max}(\Bbar) \begin{bmatrix}
\Bbar^2 & -\Bbar\\ -\Bbar & I
\end{bmatrix},
\]
	where $\lambda_{\max}(\Bbar)$ is the largest eigenvalue of $\Bbar$. From Lemma~\ref{lem:eta-bnd} we have that 
	\begin{align*}
		\lambda_{\max}(\Bbar) &\leq \max\{\|X^{-1/2}HX^{-1/2}\|,\|Y^{-1/2}VY^{-1/2}\|\}\\
		& \leq \sqrt{\|X^{-1/2}HX^{-1/2}\|^2+\|Y^{-1/2}VY^{-1/2}\|^2}\\
		& \leq \sqrt{\|X^{-1/2}HX^{-1/2}\|_F^2 + \|Y^{-1/2}VY^{-1/2}\|_F^2},
	\end{align*}
which  implies \eqref{eq:compatxicmain}.
\end{proof}

We are now in a position to prove Theorem~\ref{thm:main1}.
\begin{proof}[{Proof of Theorem~\ref{thm:main1}}]
	Theorem~\ref{thm:main1} follows directly from Theorem~\ref{thm:nesterov}, Proposition~\ref{prop:conic}, and Lemma~\ref{lem:compat1}. The function 
	\[ \S_{++}^n\times \S_{++}^n\times \S^m\ni (X,Y,Z) \mapsto -\log\det(\phi(P_g(X,Y)-Z)-\log\det(X)-\log\det(Y)\]
	has barrier parameter of $2n+m$ because it is $(2n+m)$-logarithmically homogeneous. The matching 
	lower bound
	on the barrier parameter follows from Corollary~\ref{cor:psd-dom-lb} (in Section~\ref{sec:lb}, to follow)
	since $\phi \circ P_g:\S_{++}^n\times \S_{++}^n \mapsto \S^m$ is $\S_+^m$-concave and 
	positively homogeneous of degree one.
\end{proof}

The compatibility result required for the proof of Theorem~\ref{thm:main2} is the following variation
on Lemma~\ref{lem:compat1}.
\begin{lemma}
	\label{lem:compat2}
	For any $s \geq 0$, define $G_s:\S^{n_1}_{++} \times \S^{n_2}_{++} \to \S^{n_1 n_2}$ by $G_s(X,Y) = \xi_s(X\otimes I, I \otimes \bar Y)$ where $\xi_s$ is as in \eqref{eq:xis}. Then $G_s$ is 1-compatible with $F(X,Y) = -\log \det X-\log \det Y$ where $(X,Y) \in \S^{n_1}_{++} \times \S^{n_2}_{++}$.
\end{lemma}
\begin{proof}
	First, we observe that $D^kG_s(X,Y)[H,V] = D^k\xi_s(X\otimes I,I\otimes \bar{Y})[H\otimes I,I\otimes \bar{V}]$. 
	Let $\Ybar = (1-s)(X\otimes I)+s(I\otimes \bar{Y}), \Vbar = (1-s)(H\otimes I)+s(I\otimes \bar{V})$, and $\Bbar = \Ybar^{-1/2}\Vbar\Ybar^{-1/2}$. Note that 
	the expression for $\Bbar$ in terms of $X,Y,H$, and $V$ differs from the expression in Lemma~\ref{lem:compat1}.

Following the same argument as for Lemma~\ref{lem:compat1}, 
	to prove 1-compatibility of $G_s$ with $F$, it suffices to prove that 
	\[ \lambda_{\max}(\Bbar) \leq  \sqrt{\|X^{-1/2}HX^{-1/2}\|_F^2+\|Y^{-1/2}VY^{-1/2}\|_F^2}.\]
	By applying~\ref{lem:eta-bnd} to the matrices $X\otimes I, I \otimes \bar{Y}, H\otimes I$, and $I\otimes \bar{V}$, and using the facts that $\|M\otimes I\| = \|M\|$ and $\|M\| = \|I\otimes \bar{M}\|$,  we see that 
	\begin{align*}
		\lambda_{\max}(\Bbar) &\leq \max\{\|(X\otimes I)^{-\frac{1}{2}}(H\otimes I)(X\otimes I)^{-\frac{1}{2}}\|,\|(I\otimes \bar{Y})^{-\frac{1}{2}}(I\otimes \bar{V})(I\otimes \bar{Y})^{-\frac{1}{2}}\|\}\\
		& = \max\{\|X^{-1/2}HX^{-1/2}\|,\|Y^{-1/2}VY^{-1/2}\|\}\\
		& \leq \sqrt{\|X^{-1/2}HX^{-1/2}\|_F^2+\|Y^{-1/2}VY^{-1/2}\|_F^2},
	\end{align*}
	as required.
\end{proof}
We are now in a position to prove Theorem~\ref{thm:main2}.
\begin{proof}[{Proof of Theorem~\ref{thm:main2}}]
	Theorem~\ref{thm:main2} follows directly from Theorem~\ref{thm:nesterov}, Proposition~\ref{prop:conic}, and Lemma~\ref{lem:compat2}. The lower bound on the barrier parameter follows from 
	Corollary~\ref{cor:psd-dom-lb} (in Section~\ref{sec:lb}, to follow)
	since $Q_g^\phi:\S_{++}^{n_1}\times \S_{++}^{n_2} \mapsto \S^m$ is concave and 
	positively homogeneous of degree one. 
\end{proof}

\subsection{Lower bounds on the barrier parameters}
\label{sec:lb}

In this section we establish lower bounds on the barrier parameters of convex cones related to those constructed in Theorems~\ref{thm:main1} and~\ref{thm:main2}. To establish these lower bounds we apply the following result of Nesterov.
\begin{theorem}[{\cite[Theorem 5.4.1]{nesterov2018lectures}}]
	\label{thm:lb-nesterov}
	Let $C$ be a closed convex set with nonempty interior and let $x_0 \in \interior C$. Let $p_1,\ldots,p_k$ be recession directions, i.e., satisfying 
	$x_0 + \alpha p_i \in C$ for all $\alpha \geq 0$ and all $i=1,2,\ldots,k$. 
	Let $b_1,\ldots,b_k$ be positive scalars that satisfy 
	$x_0 - b_i p_i \notin \interior C$ for $i=1,2,\ldots,k$. 
	Let $a_1,\ldots,a_k$ be positive scalars that satisfy
	$x_0 - \sum_{i=1}^{k}a_i p_i \in C$. 
	Then any self-concordant barrier for $C$ has parameter at least $\sum_{i=1}^{k}\frac{a_i}{b_i}$. 
\end{theorem}
Recall that if $K'$ is a closed convex cone we say that $h$ is $K'$-concave if $h(\lambda x + (1-\lambda)y) - \lambda h(x) - (1-\lambda)h(y) \in K'$ for all $x,y\in \dom(h)$ and $\lambda\in [0,1]$. Recall, also, that a function $h$ is positively homogeneous
of degree one if $h(\lambda x) = \lambda h(x)$ for all $\lambda>0$ and $x\in \dom(h)$. Note that the domain of a positively homogeneous function is necessarily a cone.

The following technical result (Theorem~\ref{thm:lb-technical}) shows how to use
Theorem~\ref{thm:lb-nesterov} to obtain a lower bound on the barrier
parameter of $\cl\hypo(h)$ when $h$ is positively homogeneous of degree one and concave with
respect to a closed convex cone.
\begin{theorem}
	\label{thm:lb-technical}
Let $E$ and $E'$ be finite dimensional real vector spaces. 
Suppose $K\subseteq E$ is an open convex cone and $K'\subseteq E'$ is a closed 
convex cone. Let $h:K\rightarrow E'$
be $K'$-concave and positively homogeneous of degree one.

	Let $x_1,x_2,\ldots,x_{k} \in K$ and let $a_1,\ldots,a_k$ be positive scalars such that $x_0 := \sum_{i=1}^{k}a_i x_i$. 
	Let $b_1,\ldots,b_k$ be positive scalars such that $x_0 - b_ix_i \not\in K$ for $i=1,2,\ldots,k$. 

	Let $z_1,z_2,\ldots,z_{k'}\in K'$ and let $a_1',\ldots,a_{k'}'$ be positive scalars be such that $z_0 := \sum_{i=1}^{k'}a_i'z_i \in \interior(K')$. Let $b_1',\ldots,b_{k'}'$ be positive scalars such that $z_0 - b_i'z_i \not\in \interior(K')$ for $i=1,2,\ldots,k'$. 

	Then any self concordant barrier for
	$ K_h = \cl\{(x,z)\in K\times E'\;:\; h(x) - z \in K'\}$ 
	has barrier parameter at least $\sum_{i=1}^{k}\frac{a_i}{b_i} + \sum_{i=1}^{k'}\frac{a_i'}{b_i'}$. 
\end{theorem}
\begin{proof}
	We begin with two preliminary observations. First, since $K$ is open by assumption, we note that the interior of $K_h$ is $\interior(K_h) = \{(x,z)\in K\times E'\;:\; h(x) - z \in \interior(K')\}$.
	Second, we note that if $-z\in K'$ then $(0,z)\in K_h$. To see why this is true, note that 
	if $-z\in K'$ then $(\epsilon x,z+h(\epsilon x))\in K_h$ for all $\epsilon > 0$. Since $h$ is positively 
	homogeneous of degree one, it follows that $(\epsilon x, z + \epsilon h(x))\in K_h$ for all $\epsilon >0$. 
	Since $K_h$ is closed, it follows that $(0,z)\in K_h$. 

	Let $\tau'>0$ be such that $\tau'z_0 - (h(x_0) - \sum_{i=1}^{k}a_ih(x_i)) \in K'$. 
	Such a $\tau'$ exists because $z_0\in \interior(K')$. 
	Choose some arbitrary $\tau > \tau'$ and consider the point $y_0 = (x_0,h(x_0) - \tau z_0)$. This satisfies $y_0\in \interior(K_h)$ because $z_0\in \interior(K')$ and $\tau>0$. Let $p_i = (x_i,h(x_i))\in K_h$ for $i=1,2,\ldots,k$ and let $p_i' = (0,-\tau z_i)\in K_h$ for $i=1,2,\ldots,k'$. 

	Observe that $y_0 - b_ip_i = (x_0 - b_ix_i, h(x_0) - b_ih(x_i) - \tau z_0)\not\in \interior(K_h)$ for $i=1,2,\ldots,k$ because $x_0-b_ix_i\not\in K$. Similarly $y_0 - b_i'p_i' = (x_0, h(x_0) - \tau (z_0 - b_i' z_i))\not\in\interior(K_h)$ for $i=1,2,\ldots,k'$ because
	$z_0 - b_i' z_i\not\in \interior(K')$. 

	Let $c_i = a_i$ for $i=1,2,\ldots,k$ and let $c_i' = (1-\tau'/\tau)a_i'$ for $i=1,2,\ldots,k'$. 
	Then 
 	\begin{align*}
 		y_0\! -\! \sum_{i=1}^{k}c_ip_i - \sum_{i=1}^{k'}c_i'p_i'&\! =\!
		\left(x_0 -\! \sum_{i=1}^{k}a_ix_i,h(x_0) - \tau z_0 -\! \sum_{i=1}^{k}a_ih(x_i) + (\tau-\tau')\sum_{i=1}^{k'}a_i'z_i\right)\\
		& = \left(0, h(x_0) -\! \sum_{i=1}^{k}a_ih(x_i) - \tau'z_0\right)  \in K_h
 	\end{align*}
	where the last assertion holds because $h(x_0) - \sum_{i=1}^{k}a_ih(x_i) - \tau'z_0 \in -K'$ (by our choice of $\tau'$). 
	We can now apply Theorem~\ref{thm:lb-nesterov} with the directions $p_1,\ldots,p_k,p_1',\ldots,p_{k'}'$
	to give the lower bound 
	\[ \sum_{i=1}^{k} \frac{c_i}{b_i} + \sum_{i=1}^{k'}\frac{c_i'}{b_i'} = \sum_{i=1}^{k}\frac{a_i}{b_i} + \left(1-\frac{\tau'}{\tau}\right)\sum_{i=1}^{k'} \frac{a_i'}{b_i'}\]
	on the barrier parameter of any self-concordant barrier for $K_h$. 
	Taking the limit as $\tau\rightarrow \infty$ completes the proof.
\end{proof}
We next specialize to the case when $K = \RR_{++}^{n}$ for some positive integer $n$.
\begin{proposition}
\label{prop:orthant-dom-lb}
	Let $n$ and $m$ be positive integers, and let $h:\RR_{++}^n\rightarrow \S^m$ be $\S_+^m$-concave  
	and positively homogeneous of degree one. Then any self-concordant barrier for $\cl(\hypo(h))$
	has barrier parameter at least $n+m$. 
\end{proposition}
\begin{proof}
	It is enough to construct appropriate directions $x_1,\ldots,x_n\in \RR_{++}^n$ and $z_1,\ldots,z_m\in \S_+^m$
	and apply Theorem~\ref{thm:lb-technical}. 

	Let $0<\epsilon<1$. For $i=1,2,\ldots,n$ let $x_i = \epsilon \ones + (1-\epsilon) e_i$, where $e_i$ is the $i$th standard basis vector and $\ones$ is the vector with all entries equal to one. Let $a_1 = \cdots = a_n = 1$ so that 
	$x_0 = \sum_{i=1}^{n}x_i = ((n-1)\epsilon + 1)\ones$. Let $b_1 = \cdots = b_n = ((n-1)\epsilon + 1) > 0$.  
	Then $x_0 - b_i x_i \not\in \RR_{++}^n$ for $i=1,2,\ldots,n$.

	For $i=1,2,\ldots,m$ let $z_i = e_ie_i^*$. Let $a_1' = \cdots = a_m' = 1$ so that $z_0 = \sum_{i=1}^{m}a_i'z_i = I$. 
	Let $b_1' = \cdots = b_m' = 1$. Then $z_0 - b_i'z_i = I - e_ie_i^* \not\in \S_{++}^m$. 

	It follows from Theorem~\ref{thm:lb-technical} that any self-concordant barrier for $\cl(\hypo(h))$ has barrier parameter 
	at least 
	\[ \sum_{i=1}^{n}a_i/b_i + \sum_{i=1}^{m}a_i'/b_i' = \frac{n}{(n-1)\epsilon+1} + m.\]
	Taking the limit as $\epsilon\rightarrow 0$ completes the proof. 
\end{proof}
\begin{remark}
	We note that Proposition~\ref{prop:orthant-dom-lb} (with $m=1$) specializes to give a lower bound of $2n+1$ on the barrier parameter of 
\[ \cl\left\{(x,y,z)\in \RR_{++}^n\times \RR_{++}^n \times \RR\;:\; \sum_{i=1}^{n}x_ig(y_i/x_i) \geq z\right\}\]
	for any concave function $g:\RR_{++}\rightarrow \RR$, since $h(x,y) = \sum_{i=1}^{n}P_g(x_i,y_i)$ is concave and homogeneous of degree one. Similarly, Proposition~\ref{prop:orthant-dom-lb} implies a lower bound of $n+2$ on the barrier parameter of
\[ \cl\left\{(x,y,z)\in \RR_{++}\times \RR_{++}^n\times \RR\;:\; \sum_{i=1}^{n} x g(y_i/x) \geq z\right\}\]
	for any concave function $g:\RR_{++}\rightarrow \RR$, since $h(x,y)= \sum_{i=1}^{n} P_g(x,y_i)$ is concave and homogeneous
	of degree one. In the special case $n=1$ and $g(x) = x^{\alpha}$ for $\alpha\in (0,1)$, the resulting lower bound of $3$ was established by Nesterov~\cite[Lemma 5.4.9]{nesterov2018lectures}.
\end{remark}
The case where the domain is $\S_{++}^{n_1}\times \S_{++}^{n_2}$, which is relevant for Theorems~\ref{thm:main1} and~\ref{thm:main2}, is a straightforward corollary of Proposition~\ref{prop:orthant-dom-lb}.
\begin{corollary}
	\label{cor:psd-dom-lb}
	Let $n_1, n_2$ and $m$ be positive integers, and let $h:\S_{++}^{n_1}\times \S_{++}^{n_2}\rightarrow \S^m$ be $\S_+^m$-concave and positively homogeneous of degree one. Then any self-concordant barrier for $\cl(\hypo(h))$ has
	barrier parameter at least $n_1+n_2+m$. 
\end{corollary}
\begin{proof}
	Consider $\tilde{h}:\RR_{++}^{n_1}\times \RR_{++}^{n_2}\rightarrow \S^m$ defined by 
	$\tilde{h}(x,y) = h(\diag(x),\diag(y))$ 
	where $\diag(x)$ is the diagonal matrix with diagonal entries given by $x$. Then $(x,y,Z)\in \cl\hypo(\tilde{h})$ 
	if and only if $(\diag(x),\diag(y),Z)\in \cl\hypo(h)$. 
	If $(X,Y,Z)\!\mapsto\! F(X,Y,Z)$ were a self-concordant barrier for $\cl\hypo(h)$
	with parameter less than $n_1+n_2+m$ then $(x,y,Z)\mapsto F(\diag(x),\diag(y),Z)$ would be a self-concordant barrier for 
	$\cl\hypo(\tilde{h})$ with parameter less than $n_1+n_2+m$, contradicting 
	Proposition~\ref{prop:orthant-dom-lb}. 
\end{proof}

\subsection{Proofs of the corollaries of Theorems~\ref{thm:main1} and~\ref{thm:main2}}
\label{sec:pfassembly}

We now summarize how to establish the corollaries stated in the introduction
from Theorems~\ref{thm:main1} and~\ref{thm:main2}.

The corollaries of Theorem~\ref{thm:main1} stated in Section~\ref{sec:intro} involve specializing to 
particular choices of operator concave function $g$, and particular positive linear maps $\phi$. 
Furthermore, they give an explicit description of the closure of the epigraph, based on the 
results in Appendix~\ref{sec:domainpersp}. 
\begin{proof}[Proof of {Corollary~\ref{cor:DBS}}] 
	Applying Theorem~\ref{thm:main1} with the operator concave function $g(x) = \log(x)$ and 
	the positive linear map $\phi(X) = \tr(X)$ gives the barrier
		$ -\log(-D_{BS}(X|Y)-z) - \log\det(X)-\log\det(Y)$
		with (optimal) parameter $2n+1$ for 
	$\cl \{(X,Y,z)\in \S_{++}^n\times \S_{++}^n\times \RR\;:\; -D_{BS}(X|Y) \geq z\}$.
		Since $g(0)=-\infty$, the explicit description of this closure is given by
		item (ii) of Theorem~\ref{thm:maincl1}. 
	 Composing with the linear map $z\mapsto -z$ gives the barrier 
		$-\log(z-D_{BS}(X|Y)) - \log\det(X)-\log\det(Y)$
		for the set $\epi(D_{BS})$.
\end{proof}

\begin{proof}[Proof of {Corollary~\ref{cor:R}}] 
	Observe that 
	$\hat{Q}_\alpha(X|Y) = \phi(P_g(X,Y))$
	where $g(x) = x^{1-\alpha}$ and $\phi(X) = \tr(X)$. 
	
	If $\alpha\in [0,1]$ then $g(x) = x^{1-\alpha}$ is operator concave and so $\hat{Q}_{\alpha}$ is 
	concave. 
	Applying Theorem~\ref{thm:main1} gives the self-concordant barrier
	$-\log(z-\hat{Q}_{\alpha}(X|Y))-\log\det(X)-\log\det(Y)$
	with (optimal) parameter $2n+1$ for 
	\begin{multline*}
		\cl \{(X,Y,z)\in \S_{++}^n\times \S_{++}^n\times \RR\;:\; \hat{Q}_\alpha(X|Y) \geq z\} = \\
	\{(X,Y,z)\in \S_{+}^n\times \S_{+}^n\times \RR\;:\; \hat{Q}_\alpha(X|Y) \geq z\}.\end{multline*}
	Here, the last equality follows from item (i) of Theorem~\ref{thm:maincl1} since $g(0)=0 = \hat{g}(0)$. 

	If $\alpha\in [-1,0]\cup (1,2]$ then $x^{1-\alpha}$ is operator convex and so $\hat{Q}_\alpha$ is convex. 
	Applying Theorem~\ref{thm:main1} with the operator concave function $g(x) = -x^{1-\alpha}$
	gives the barrier
		$-\log(-\hat{Q}_\alpha(X|Y)-z) - \log\det(X)-\log\det(Y)$
		with (optimal) parameter $2n+1$ for 
	$\cl \{(X,Y,z)\in \S_{++}^n\times \S_{++}^n\times \RR\;:\; -\hat{Q}_\alpha(X|Y) \geq z\}$.
		If $\alpha\in (1,2]$ then $g(0^+)=-\infty$ and $\hat{g}(0) > -\infty$. Therefore 
		the explicit description of this closure is given by
		item (ii) of Theorem~\ref{thm:maincl1}. If $\alpha \in [-1,0)$ then $g(0) > -\infty$ and $\hat{g}(0^+) = -\infty$, and so the explicit description of the closure is given by item (iii) of Theorem~\ref{thm:maincl1}.
		Composing with the linear map $z\mapsto -z$ gives the barrier 
		$ -\log(z-\hat{Q}_\alpha(X|Y)) - \log\det(X)-\log\det(Y)$
		for the set $\epi(\hat{Q}_\alpha)$.
\end{proof}

Next we establish the corollaries of Theorem~\ref{thm:main2}, regarding the quantum relative entropy cone and the functions $Q_\alpha$. 
\begin{proof}[{Proof of Corollary~\ref{cor:qre}}]
	Consider the operator concave function $g(x) = \log(x)$ and 
	the positive linear map $\p:\S^{n^2}\rightarrow \RR$ 
	with the property that $\p(X\otimes \bar{Y}) = \tr(XY)$.
	Then $Q_g^\p(X|Y) = -D(X|Y)$. This is the case (see, e.g.,~\cite{effros2009matrix}) because
	\begin{align*}
		P_{\log}(X\otimes I, I\otimes \bar{Y}) & = (X\otimes I)^{1/2}\log((X\otimes I)^{-1/2}(I\otimes \bar{Y})(X\otimes I)^{-1/2})(X\otimes I)^{1/2}\\
		&= (X^{1/2}\otimes I) \log(X^{-1}\otimes \bar{Y}) (X^{1/2}\otimes I)\\
		&= -X\log(X)\otimes I + X \otimes \overline{\log(Y)}.
	\end{align*}
	Here we have used the fact that $(X\otimes I)^{\alpha} = X^{\alpha}\otimes I$ and that for positive definite matrices $A$ and $B$, $\log(A^{-1}\otimes B) = -\log(A)\otimes I + I\otimes \log(B)$. We have 
	also used the property that $\log(\bar{Y}) = \overline{\log(Y)}$ for any positive definite $Y$. 

	Applying Theorem~\ref{thm:main2} with this choice of $g$ and $\p$
	gives the barrier
		\[-\log(-D(X|Y)-z) - \log\det(X)-\log\det(Y)\]
		with (optimal) parameter $2n+1$ for 
	$\cl \{(X,Y,z)\in \S_{++}^n\times \S_{++}^n\times \RR\;:\; -D(X|Y) \geq z\}$.
		Since $g(0^+)=-\infty$ and $\hat{g}(0) = 0$, the explicit description of this closure is given by
		item (ii) of Theorem~\ref{thm:maincl2}. 
		Composing with the linear map $z\mapsto -z$ gives the barrier 
		$-\log(z-D(X|Y)) - \log\det(X)-\log\det(Y)$
		for the set $\epi(D)$.
\end{proof}

\begin{proof}[{Proof of Corollary~\ref{cor:Qalpha}}]
	Consider the function $g(x) = x^{1-\alpha}$ and the positive linear map
	$\p:\S^{n^2}\rightarrow \RR$ 
	with the property that $\p(X\otimes \bar{Y}) = \tr(XY)$.
	Then $Q_g^\p(X|Y) = Q_\alpha(X|Y)$. This is the case because 
	\begin{align*}
		P_{g}(X\otimes I, I\otimes \bar{Y}) & = (X\otimes I)^{1/2}((X\otimes I)^{-1/2}(I\otimes \bar{Y})(X\otimes I)^{-1/2})^{1-\alpha}(X\otimes I)^{1/2}\\
		& = (X^{1/2}\otimes I)(X^{-(1-\alpha)}\otimes \bar{Y}^{1-\alpha})(X^{1/2}\otimes I)
		= X^{\alpha}\otimes \bar{Y}^{1-\alpha}.
	\end{align*}
	Here we have used the fact that $(A\otimes B)^{\alpha} = A^{\alpha}\otimes B^{\alpha}$ for
	$A,B \in\S_{++}^n$.

	If $\alpha\in [0,1]$, $g(x) = x^{1-\alpha}$ is operator concave, so applying Theorem~\ref{thm:main2}
	gives 
	$-\log(Q_\alpha(X|Y) - z) - \log\det(X) - \log\det(Y)$
	as a barrier, with (optimal) parameter $2n+1$, for $\cl\{(X,Y,z)\in \S_{++}^n\times \S_{++}^n\times \RR\;:\; Q_\alpha(X|Y) \geq z\}$. 
	Since $g(0) = \hat{g}(0) = 0$, the explicit description of this closure is given by 
	item (i) of Theorem~\ref{thm:maincl2}. 

	If $\alpha\in [-1,0]\cup (1,2]$ then $x^{1-\alpha}$ is operator convex and so $Q_\alpha$ is convex. 
	Applying Theorem~\ref{thm:main2} with the operator concave function $g(x) = -x^{1-\alpha}$
	gives the barrier
		$-\log(-Q_\alpha(X|Y)-z) - \log\det(X)-\log\det(Y)$
		with (optimal) parameter $2n+1$ for 
	$\cl \{(X,Y,z)\in \S_{++}^n\times \S_{++}^n\times \RR\;:\; -Q_\alpha(X|Y) \geq z\}$.
		If $\alpha\in (1,2]$ then $g(0^+)=-\infty$ and $\hat{g}(0) > -\infty$. Therefore 
		the explicit description of this closure is given by
		item (ii) of Theorem~\ref{thm:maincl2}. If $\alpha \in [-1,0)$ then $g(0) > -\infty$ and $\hat{g}(0^+) = -\infty$, and so the explicit description of the closure is given by item (iii) of Theorem~\ref{thm:maincl2}.
		Composing with the linear map $z\mapsto -z$ gives the barrier 
		$-\log(z-Q_\alpha(X|Y)) - \log\det(X)-\log\det(Y)$
		for the set $\epi(Q_\alpha)$.
\end{proof}


\section{Discussion}
\label{sec:discussion}

We conclude by discussing natural questions related to self-concordant 
barriers for convex cones related to those studied in this paper, and topics for further research 
related to the quantum relative entropy cone.

\paragraph{Self-concordant barriers for generalizations of the Lieb-Ando functions}
For a fixed invertible $n\times n$ matrix $K$, let $p\geq q$ and let $s>0$. The functions 
\begin{equation}
	\label{eq:psi3}
	(A,B) \mapsto \tr\left[(B^{q/2}K^*A^{p}KB^{q/2})^{s}\right],
\end{equation}
defined on a pair of positive definite $n\times n$ matrices, are
\begin{itemize}
\item jointly concave if $0\leq q\leq p\leq 1$ and $0<s\leq 1/(p+q)$;
\item jointly convex if $-1\leq q\leq p \leq 0$ and $s>0$;
\item jointly convex if $-1\leq q\leq 0$, $1\leq p \leq 2$, $(p,q)\neq (1,-1)$ and $s\geq 1/(p+q)$.
\end{itemize}
This result, in its full generality, is due to Zhang~\cite{zhang2020wigner}. When $s=p+q=1$, 
the concave case is due to Lieb~\cite{lieb}, and the convex case is due to Ando~\cite{ando1979concavity}.
For general $s$, an important special case of these functions are the sandwiched R\'enyi divergences~\cite{MDSFT13,WWY13}
\[ (A,B) \mapsto \tr\left[(B^{\frac{1-t}{2t}}K^*AKB^{\frac{1-t}{2t}})^{t}\right]\]
which correspond to the case $p=1$, $q=1/t-1$, and $s=1/(p+q)=t$. These 
are jointly concave for $t\in [1/2,1]$ and jointly convex for $t\geq 1$ \cite{frank2013monotonicity,beigi2013sandwiched}.

It would be interesting to find efficiently computable 
optimal self-concordant barriers for the closures of the hypo/epigraphs of the functions defined in~\eqref{eq:psi3}
for the full range of parameters where they are concave/convex. 
We note that in the special cases where $-1\leq p\leq 1/2$, $q=1-p$ and $s=1$, efficiently computable
logarithmically homogeneous self-concordant barriers with parameter $2n+1$ for these cones can 
be obtained as a corollary of Theorem~\ref{thm:main2}. 

\paragraph{Trace functions}
Suppose $I\subseteq \RR$ is an interval and $f:I\to \RR$ is a concave function. The 
trace function $\tr f(X) = \sum_{i=1}^{n}f(\lambda_i(X))$, 
defined on $n\times n$ Hermitian matrices $X$ with eigenvalues in $I$ is also a concave function. It then 
follows that the (scalar) perspective $(x,Y)\mapsto \tr(xf(Y/x))$ is concave and positively
homogeneous of degree one. 
In the special case where $f:(0,\infty)\rightarrow \RR$ is operator concave, it follows from Theorem~\ref{thm:main2} that $-\log(\tr(xf(Y/x))-z) - \log\det(Y) - \log(x)$ is a self-concordant barrier 
with (optimal) barrier parameter $n+2$ for the closure of the hypograph of $(x,Y) \mapsto \tr(x f(Y/x))$.
To see why, observe that choosing $n_1=1$, $n_2=n$, and $\phi(A) = \tr(\bar{A})$ in Theorem~\ref{thm:main2} gives
\[ \phi(P_f(x\otimes I, 1 \otimes \bar{Y})) = \phi (P_f(xI,\bar{Y})) = \tr(\overline{xf(\bar{Y}/x)}) = \tr(xf(Y/x)).\]
This consequence of Theorem~\ref{thm:main2} generalizes~\cite[Proposition 1]{coey2022conic}, which deals with the special 
case where $f$ has operator monotone derivative.\footnote{A closely related (earlier) result of Faybusovich and Tsuchiya~\cite{faybusovich2017matrix} gives a self-concordant barrier with parameter $n+1$ for the epigraph of $X\mapsto \tr(f(X))$ (without the perspective), where $f$ has operator monotone derivative.} (This is a generalization because for a function $f:(0,\infty)\rightarrow \RR$, having operator monotone derivative implies being operator convex~\cite[Exercise V.3.14]{bhatia2013matrix}, but the converse does not hold in general, as the example $f(x) = x^{-1}$ shows.)
For general concave functions that are not operator concave, it is not clear how to construct efficiently 
computable optimal self-concordant barriers for the hypographs of the associated trace functions.

\paragraph{Conic optimization with the quantum relative entropy cone}
With the availability of both an optimal self-concordant barrier for the
quantum relative entropy cone, and software for optimization over
nonsymmetric cones that already implements this
barrier~\cite{coey2022performance,karimi2019domain,papp2022alfonso}, the time seems ripe for a more in-depth
study of conic optimization with respect to the quantum relative entropy cone.
This could include: studying the modeling power of lifted representations using
the quantum relative entropy cone, along the lines of the study of
lifted representations using the positive semidefinite
cone~\cite{fawzi2022lifting}; studying the facial structure of the quantum
relative entropy cone; and studying error bounds for the quantum relative
entropy cone, which would generalize the already subtle error bounds for the
exponential cone~\cite{lindstrom2020error}.

\appendix

\section{Integral representations of operator concave functions}
\label{sec:intrepoperatorconcave}

The following theorem gives an integral representation of operator monotone functions on $(0,\infty)$. 
\begin{theorem}[{L{\"o}wner's theorem, see \cite[Theorem 4]{logapprox}}]
\label{thm:intrepopmonotone}
If $h:(0,\infty) \to \RR$ is an operator monotone function, then there is a finite measure $\mu$ supported on $[0,1]$ such that
\begin{equation}
\label{eq:intrepopmonotone}
h(x) = h(1) + \int_{0}^{1} \frac{x-1}{1+s(x-1)} d\mu (s)  \qquad \forall x > 0.
\end{equation}
\end{theorem}
We can deduce from this theorem the following integral representation of operator concave functions.
\begin{theorem}
\label{thm:intrepopconcave}
If $g:(0,\infty) \to \RR$ is an operator concave function, then there is a finite measure $\mu$ supported on $[0,1]$ such that
\begin{equation}
\label{eq:intrepopconcave}
g(x) = g(1) + g'(1)(x-1) - \int_{0}^{1} \frac{(x-1)^2}{1+s(x-1)} d\mu (s)  \qquad \forall x > 0.
\end{equation}
\end{theorem}
\begin{proof}
	It is known that $g(x)$ is operator concave if, and only if, the function $h(x)= -(g(x)-g(1))/(x-1)$ is operator monotone, see \cite[Corollary 2.7.8]{hiainotes}. (Here $h(1)$ is interpreted as $-g'(1)$ via taking the appropriate limit.) The result follows from applying Theorem \ref{thm:intrepopmonotone}.
\end{proof}

\subsection{Directional derivatives of noncommutative perspectives}

Let $g:(0,\infty)\to \RR$ be an operator concave function, and let
\[
P_g(X,Y) = X^{1/2} g(X^{-1/2} Y X^{-1/2}) X^{1/2},
\]
be its noncommutative perspective. For $s \in [0,1]$, let 
\begin{equation}
\label{eq:xisappnew123}
\xi_s(X,Y) = -(Y-X)(X+s(Y-X))^{-1}(Y-X)
\end{equation}
be the noncommutative perspective of $x\mapsto -\frac{(x-1)^2}{1+s(x-1)}$. It follows immediately from \eqref{eq:intrepopconcave} that
\begin{equation}
\label{eq:intrepPgXY123}
P_g(X,Y) = g(1) X + g'(1)(Y-X) + \int_{0}^{1} \xi_s(X,Y) d\mu(s).
\end{equation}

The next theorem shows that the directional derivatives of $P_g$ can be expressed as an integral of the directional derivatives of $\xi_s$.
\begin{theorem}
\label{thm:intrepdirectionalderiv}
	Let $g:(0,\infty)\to \RR$ be an operator concave function with representing measure $\mu$ in~\eqref{eq:intrepopconcave}. Let $P_g$ be its noncommutative perspective. Let $X,Y\in \S_{++}^n$ and $H,V\in \S^n$. If $D^k P_g(X,Y)[H,V]$ is the $k$th directional derivative of $P_g$ at $(X,Y)$ in the direction $(H,V)$ then, for any $k\geq 1$,
\begin{equation}
\label{eq:intrepdirectionalderiv}
	D^k P_g(X,Y)[H,V] =  \begin{cases} g(1)H + g'(1)(V-H) +  \int_{0}^{1} D \xi_s(X,Y)[H,V] \;d\mu(s) & \textup{if $k=1$}\\
	\int_{0}^{1} D^k \xi_s(X,Y)[H,V] \;d\mu(s) & \textup{if $k \geq 2$}.\end{cases}
\end{equation}
\end{theorem}
\begin{proof}
Using the integral representation~\eqref{eq:intrepPgXY123}, for all $t$ such that $X+tH\pd 0$ and $Y+tV \pd 0$ we have
\[
\begin{aligned}
P_g(X+tH, Y+tV) &= g(1)(X+tH) + g'(1)(Y-X+t(V-H))\\
& \qquad + \int_{0}^{1} \xi_s(X+tH, Y+tV) d\mu(s).
\end{aligned}
\]
	By (repeatedly) differentiating the above identity with respect to $t$ we get:
\begin{equation}
\label{eq:DPgXYHVpre}
	D P_g(X, Y)[H,V] = g(1) H + g'(1)(V-H) + \left.\frac{d}{dt}\right|_{t=0} \int_{0}^{1} \xi_s(X+tH, Y+tV) d\mu(s)
\end{equation}
and, for $k\geq2$, 
\begin{equation}
	\label{eq:DPgkpre}
	D^k P_g(X, Y)[H,V] = \left.\frac{d^k}{dt^k}\right|_{t=0} \int_{0}^{1} \xi_s(X+tH, Y+tV) d\mu(s).
\end{equation}
Our main task is to justify that we can swap the order of integration and differentiation in the right-hand sides of the equations above.  By an application of the dominated convergence theorem (see, e.g.,~\cite[Corollary 5.9]{bartle2014elements}), it is enough to show that, for each $k\geq 1$, the partial derivative $(\partial^k/\partial t^k) \zeta(t,s)$ of $\zeta(t,s) = \xi_s(X+tH,Y+tV)$ is uniformly bounded in a neighborhood of $t=0$ for all $s \in [0,1]$, i.e., there are constants $\eps>0$ and  $C_k > 0$ such that $\|(\partial^k/\partial t^k) \zeta(t,s)\| \leq C_k$ for all $(t,s) \in [-\eps,\eps] \times [0,1]$.

We proceed as follows. Since $X,Y \pd 0$, there exist $\eps > 0$ and $\delta > 0$ such that $X+tH, Y+tV \psd \delta I$ for all $t \in [-\eps,\eps]$. Hence we can write
\[
(1-s)(X+tH) + s(Y+tV) \; \psd \; \delta I \;\; \forall (t,s) \in [-\eps,\eps] \times [0,1].
\]
By continuity of $A\mapsto \lambda_{\min}(A)$, the above will also hold when $s$ is in an open neighborhood of $[0,1]$, i.e., there is a small enough $\eta > 0$ such that:
\[
(1-s)(X+tH) + s(Y+tV) \; \psd \; (\delta/2) I \;\; \forall (t,s) \in [-\eps,\eps] \times [-\eta,1+\eta].
\]
Since
\[
\begin{aligned}
	&\zeta(t,s) = \xi_s(X+tH,Y+tV)\\
&= -(Y-X + t(V-H))\Bigl((1-s)(X+tH) + s(Y+tV)\Bigr)^{-1} (Y-X + t(V-H))
\end{aligned}
\]
	this shows that $\zeta(t,s)$ is well-defined and $C^{\infty}$ on $[-\eps,\eps]\times [-\eta,1+\eta]$. Thus this means that all its derivatives are bounded on $[-\eps/2,\eps/2] \times [0,1]$, which is what we need to (repeatedly) swap the integration and differentiation operations. Hence from \eqref{eq:DPgXYHVpre} we get
\begin{equation}
\label{eq:DPgXYHVafter}
	D P_g(X, Y)[H,V] = g(1) H + g'(1)(V-H) + \int_{0}^{1} D \xi_s(X, Y)[H,V] d\mu(s)
\end{equation}
and, for $k\geq 2$, 
\[ D^k P_g(X, Y)[H,V] = \int_{0}^{1} D^k \xi_s(X, Y)[H,V] d\mu(s),\]
as desired.
\end{proof}

\begin{remark}
\label{rem:intrepdirectionalderivtensor}
If we let $\Gamma(X,Y) = P_g(X\otimes I, I \otimes \bar Y)$ and $G_s(X,Y) = \xi_s(X\otimes I, I \otimes \bar Y)$, then it follows from \eqref{eq:intrepdirectionalderiv} that for $k\geq 2$
\[
D^k \Gamma(X,Y) [H,V] = \int_{0}^{1} D^k G_s(X,Y)[H,V] d\mu(s)
\]
since $D^k \Gamma(X,Y) [H,V] = D^k P_g(X\otimes I, I\otimes \bar Y)[H\otimes I, I \otimes \bar V]$, and $D^k G_s(X,Y)[H,V] = D^k\xi_s(X\otimes I, I \otimes \bar Y)[H\otimes I, I \otimes \bar V]$.
\end{remark}

We conclude the section with formulas for the directional derivatives of the $\xi_s$. 
\begin{lemma}
	\label{lem:diff-facts}
	For $s\in [0,1]$ let $\xi_s$ be as defined in \eqref{eq:xisappnew123}.
	Let $X,Y\in \S_{++}^n$ and $H,V\in \S^n$ be fixed matrices, and let $Y_s = (1-s)X+sY$, $V_s = (1-s)H+sV$, and $B_s = Y_s^{-1/2}V_sY_s^{-1/2}$. Then for $k \geq 1$,
	\[
		D^k \xi_s(X,Y)[H,V] = -k{!}\begin{bmatrix} Y_s^{-\frac{1}{2}}(Y-X)\\ Y_s^{-\frac{1}{2}}(V-H) \end{bmatrix}^* \begin{bmatrix} (-B_s)^k & (-B_s)^{k-1}\\ (-B_s)^{k-1} & W_k \end{bmatrix} \begin{bmatrix} Y_s^{-\frac{1}{2}}(Y-X)\\ Y_s^{-\frac{1}{2}}(V-H) \end{bmatrix}
	\]
	where
	\[
		W_k = \begin{cases} 0 & \text{ if } k=1\\
			(-B_s)^{k-2} & \text{ if } k \geq 2.
 	\end{cases}
	\]
\end{lemma}
\begin{proof}
	Recall that $D^k\xi_s(X,Y)[H,V]$ is $k{!}$ times the coefficient of $t^k$ in the Taylor expansion
	of $\xi_s(X+tH,Y+tV)$ about $t=0$. For sufficiently small $t$, 
	\begin{align*}
		\xi_s&(X+tH,Y+tV) = -(Y-X+t(V-H))(\Ybar + t\Vbar)^{-1}(Y-X+t(V-H))\\
		& = -(Y-X+t(V-H))\Ybar^{-1/2}(I+t\Bbar)^{-1}\Ybar^{-1/2}(Y-X+t(V-H))\\
		& = -(Y-X+t(V-H))\Ybar^{-1/2}\left(\sum_{k=0}^{\infty}t^{k}(-\Bbar)^k\right)\Ybar^{-1/2}(Y-X+t(V-H)).
	\end{align*}
	If $k\geq 2$, the coefficient of $t^k$ is then
	\begin{multline*}
		-(Y-X)\Ybar^{-1/2}(-\Bbar)^{k}\Ybar^{-1/2}(Y-X) - (Y-X)\Ybar^{-1/2}(-\Bbar)^{k-1}\Ybar^{-1/2}(V-H)\\
		-(V-H)\Ybar^{-1/2}(-\Bbar)^{k-1}\Ybar^{-1/2}(Y-X)
		- (V-H)\Ybar^{-1/2}(-\Bbar)^{k-2}\Ybar^{-1/2}(V-H).
	\end{multline*}
	(If $k=1$ then the last term is replaced by zero.)
	We obtain the stated expressions for $D^k\xi_s(X,Y)[H,V]$ by 
	rewriting in matrix form and multiplying by $k{!}$. 
\end{proof}


\section{Domains of perspectives of operator concave functions}
\label{sec:domainpersp}

Let $g:(0,\infty)\to \RR$ be an operator concave function, and let $P_g$ be its noncommutative perspective
$P_g(X,Y) = X^{1/2} g(X^{-1/2} Y X^{-1/2}) X^{1/2}$ 
defined for positive definite matrices $X$ and $Y$.
The purpose of this section is to give an explicit description of the closure of the hypograph of $P_g$, and related functions.

Recall that to $g$ we can associate the transpose function $\hat{g}(x) = xg(1/x)$ defined on $(0,\infty)$ which satisfies $P_g(X,Y) = P_{\hat g}(Y,X)$ for all $X,Y \pd 0$. The closure of the hypograph of $P_g$ will depend on the behavior of $g$ and $\hat{g}$ as $x\to 0$. We write $g(0^+) = \lim_{x\to 0} g(x)$ and similarly for $\hat{g}(0^+)$. Note that since $g$ and $\hat{g}$ are concave functions we have $g(0^+),\hat{g}(0^+) \in \RR\cup \{-\infty\}$.

The following are the two main theorems of this appendix.
\begin{theorem}
\label{thm:maincl1}
	Assume $g:(0,\infty) \to \RR$ is operator concave, and let $P_g$ be its noncommutative perspective.
	Let $\phi:\S^n\to \S^m$ be a positive linear map, and let
	\[
	\cD = \{(X,Y) \in \S^n_+ \times \S^n_+ : \phi(P_g(X+\epsilon I, Y+\epsilon I)) \text{ is bounded below as } \epsilon\downarrow 0\}.
	\]
	Then for $(X,Y) \in \cD$, $\lim_{\epsilon\downarrow0} \phi(P_g(X+\epsilon I, Y+\epsilon I))$ exists, and, for $X,Y\pd 0$ coincides with $\phi(P_g(X,Y))$.
	Extending $\phi \circ P_g$ to $\cD$ in this way, we have
	\begin{multline*}
	\cl \{(X,Y,Z) \in \S^n_{++} \times \S^n_{++} \times \S^m : \phi(P_g(X,Y)) \psd Z\}\\ = \{(X,Y,Z) \in \cD \times \S^n : \phi(P_g(X,Y)) \psd Z\}.
	\end{multline*}
	
	Furthermore, assuming $\phi$ satisfies\footnote{This condition is obviously satisfied e.g., for the linear maps $\phi(X)=X$ and $\phi(X)=\tr X$.}
	\begin{equation}
	\label{eq:assumptionphi}
	\exists c > 0 : \tr \phi(X) \geq c \tr X \qquad \forall X \psd 0,
	\end{equation}
	the set $\cD$ is equal to the following, according to the limits $g(0^+)$ and $\hat g(0^+)$:
	
\noindent (i) If $g(0^+) > -\infty$ and $\hat{g}(0^+) > -\infty$, then $\cD = \S^n_+ \times \S^n_+$\\
(ii) If $g(0^+) = -\infty$ and $\hat{g}(0^+) > -\infty$ then $\cD = \{(X,Y) \in \S^n_+ \times \S^n_+ : X \ll Y\}$\\
(iii) If $g(0^+) > -\infty$ and $\hat{g}(0^+) = -\infty$ then $\cD = \{(X,Y) \in \S^n_+ \times \S^n_+ : Y \ll X\}$\\
(iv) If $g(0^+) =  \hat{g}(0^+)=-\infty$ then $\cD = \{(X,Y) \in \S^n_+ \times \S^n_+ : \ker(X) = \ker(Y)\}$.
\end{theorem}

\begin{remark}
We note that although $P_g$ is continuous on $\S^n_{++} \times \S^n_{++}$, it is in general \emph{not} continuous on the domain $\cD$. 
	Consider the case $g(x) = x^{1/2}$ for which $g(0) = \hat{g}(0) = 0$. Thus $\lim_{\epsilon \downarrow 0} P_g(X+\epsilon I, Y+\epsilon I)$ is well-defined for all $X,Y \psd 0$. If $v,w \in \CC^n$ are unit normed, one can show that (see e.g., \cite[Remark 2.3]{fawzi2021defining})
\[
P_g(vv^*,ww^*) = \lim_{\epsilon \downarrow 0} P_g(vv^* + \epsilon I, ww^* + \epsilon I) = \begin{cases} 0 & \text{ if } v\neq w\\
vv^* & \text{ otherwise.}
\end{cases}
\]

\end{remark}

The second main result concerns functions of the form $\phi(P_g(X\otimes I, I\otimes \bar Y))$.

\begin{theorem}
\label{thm:maincl2}
Let $g:(0,\infty)\to \RR$ be an operator concave function, and let $P_g$ be its noncommutative perspective.
Let $\p:\S^{n^2}\to \RR$ be the positive linear map such that $\p(X\otimes \bar Y) = \tr(XY)$ for all $X,Y \in \S^n$, and define
\[
Q_g(X|Y) = \p(P_g(X\otimes I, I \otimes \bar Y))
	\qquad\textup{for $X,Y \pd 0$}.
	\]
Let
	$
	\cD = \{(X,Y) \in \S^n_+ \times \S^n_+ : Q_g(X+\epsilon I|Y+\epsilon I) \text{ is bounded below as } \epsilon\downarrow 0\}.
	$
	 Then for $(X,Y) \in \cD$, $\lim_{\epsilon\downarrow0} Q_g(X+\epsilon I| Y+\epsilon I)$ exists, and, for $X,Y\pd 0$,  coincides with $Q_g(X|Y)$.
	Extending $Q_g$ to $\cD$ in this way, we have
	\[
	\cl \{(X,Y,z) \in \S^n_{++} \times \S^n_{++} \times \RR : Q_g(X|Y) \geq z\} = \{(X,Y,z) \in \cD \times \RR : Q_g(X|Y) \geq z\}.
	\]

Furthermore, according to the limits $g(0^+)$ and $\hat g(0^+)$, the set $\cD$ is given by:
	
\noindent (i) If $g(0^+) > -\infty$ and $\hat{g}(0^+) > -\infty$, then $\cD = \S^n_+ \times \S^n_+$\\
(ii) If $g(0^+) = -\infty$ and $\hat{g}(0^+) > -\infty$ then $\cD = \{(X,Y) \in \S^n_+ \times \S^n_+ : X \ll Y\}$\\
(iii) If $g(0^+) > -\infty$ and $\hat{g}(0^+) = -\infty$ then $\cD = \{(X,Y) \in \S^n_+ \times \S^n_+ : Y \ll X\}$\\
(iv) If $g(0^+) =\hat{g}(0^+) =  -\infty$ then $\cD = \{(X,Y) \in \S^n_+ \times \S^n_+ : \ker(X) = \ker(Y)\}$.
\end{theorem}

\subsection{Preliminaries}

We recall here some results from convex analysis needed to establish Theorems \ref{thm:maincl1} and \ref{thm:maincl2}.

First, it is well-known that if $g:(0,1)\to \RR$ is concave then $\lim_{\epsilon\downarrow 0} g(\epsilon)$ exists in $\RR \cup \{-\infty\}$. A corollary of the above to matrix concave functions is the following:
\begin{proposition}
Let $\xi:(0,1)\to \S^m$ be $\S^m_+$-concave. Then $\lim_{\epsilon\downarrow0} \xi(\epsilon)$ exists if, and only if, $\xi(\epsilon)$ is bounded below (in the positive semidefinite sense) as $\epsilon\downarrow 0$.
\end{proposition}
\begin{proof}
The implication $\implies$ is obviously true.
	Assume conversely that $\xi(\epsilon)$ is bounded below as $\epsilon\downarrow0$. Then for any $v \in \CC^n$, the same is true for the concave scalar-valued functions $\xi_v(\epsilon) = v^* \xi(\epsilon) v$, and thus $\lim_{\epsilon\downarrow0} \xi_v(\epsilon)$ is finite for all $v \in \CC^n$. This means that the entries of $\xi(\epsilon)$ are all convergent since the $(p,q)$ entries of the real and imaginary parts of $\xi(\epsilon)$ satisfy, respectively,
\begin{align}
\label{eq:ReHermitian}
	\text{Re} \, \xi(\epsilon)_{pq} &= \frac{1}{2}\left[(e_p+e_q)^* \xi(\epsilon) (e_p+e_q) - e_p^* \xi(\epsilon) e_p - e_q^* \xi(\epsilon) e_q\right]\quad\textup{and}\\
\label{eq:ImHermitian}
	\text{Im} \, \xi(\epsilon)_{pq} &= \frac{1}{2} \left[(e_p-ie_q)^* \xi(\epsilon) (e_p-ie_q) - e_p^* \xi(\epsilon) e_p - e_q^* \xi(\epsilon) e_q\right]
\end{align}
	where $e_p$ is the $p$th standard basis vector in $\CC^m$.
\end{proof}

Recall that if $g$ is a concave function, then the \emph{closure} of $g$ is the function whose hypograph is the closure of the hypograph of $g$ \cite[Page 52]{rockafellar}; alternatively it is the pointwise smallest, upper semi-continuous function that upper bounds $g$. The following fact about the closure of concave functions $g$ will be important to us.

\begin{theorem}[{\cite[Theorem 7.5]{rockafellar}}]
\label{thm:clscalar}
Let $g:C\to \RR$ be a concave function defined on an open convex set $C \subset \RR^n$. Let $e$ be an arbitrary point in $C$ and let 
\[
\cD = \{x \in \cl C : g( (1-\epsilon) x + \epsilon e ) \text{ bounded below as } \epsilon\downarrow0 \}.
\]
Then for $x \in \cD$, $\lim_{\epsilon\downarrow0} g( (1-\epsilon) x + \epsilon e )$ exists and coincides with $g(x)$ when $x \in C$. Extending $g$ to $\cD$ in this way, we have
\[
\cl \left\{ (x,z) \in C \times \RR : g(x) \geq z \right\} = \left\{(x,z) \in \cD\times \RR : g(x) \geq z\right\}.
\]
\end{theorem}

We now prove a general version dealing with $\S^n_+$-concave functions.
\begin{theorem}
\label{thm:clmatrix}
Let $\xi:C\to \S^m$ be a $\S^m_+$-concave function defined on an open convex set $C \subset \RR^n$. Let $e$ be an arbitrary point in $C$ and let \[
\cD = \{x \in \cl C : \xi( (1-\epsilon) x + \epsilon e ) \text{ bounded below as } \epsilon\downarrow0 \}.
\]
Then for $x \in \cD$, $\lim_{\epsilon\downarrow0} \xi( (1-\epsilon) x + \epsilon e )$ exists and coincides with $\xi(x)$ when $x \in C$. Extending $\xi$ to $\cD$ in this way, we have
\begin{equation}
\label{eq:clxZCSm}
\cl \left\{ (x,Z) \in C \times \S^m : \xi(x) \psd Z \right\} = \left\{(x,Z) \in \cD\times \S^m : \xi(x) \psd Z\right\}.
\end{equation}
\end{theorem}
\begin{proof}
First we show that the right hand side of \eqref{eq:clxZCSm} is closed.
Let $(x_k,Z_k) \in \cD \times \S^m$ be a sequence converging to $(x,Z)$ such that $\xi(x_k) \psd Z_k$. We show that necessarily $x \in \cD$, and $\xi(x) \psd Z$. For any $v \in \CC^m$, consider the real-valued concave function $\xi_v(x) = v^* \xi(x) v$. Applying Theorem \ref{thm:clscalar} to the function $\xi_v$ and the sequence $(x_k,z_k=v^* Z_k v) \to (x,v^* Z v)$ living in $\hypo(\xi_v)$, we get that $x \in \cD_v$ where
\[
\cD_v = \left\{ \bar x \in C : v^* \xi((1-\epsilon) \bar x + \epsilon e) v \text{ bounded below as } \epsilon \downarrow 0 \right\}
\]
and that 
\begin{equation}
\label{eq:limitvxiv>=vZv}
\lim_{\epsilon \downarrow 0} v^* \xi((1-\epsilon) x + \epsilon e) v \geq v^* Z v.
\end{equation}
Since $v^* \xi((1-\epsilon) x + \epsilon e) v$ has a limit as $\epsilon\downarrow 0$ for all $v$, it follows from \eqref{eq:ReHermitian} and \eqref{eq:ImHermitian} that the matrix $\xi((1-\epsilon) x + \epsilon e)$ has a limit as $\epsilon \downarrow 0$. Furthermore, from \eqref{eq:limitvxiv>=vZv} (true for all $v$) this limit is $\psd Z$.

Now we show that the right hand side is indeed the closure of the hypograph of the left-hand side. It suffices to take for $(x,Z) \in \cD \times \S^m$ such that $\xi(x) \psd Z$, the sequence $(x_k,Z_k)$ where $x_k = (1-k^{-1}) x + k^{-1} e$ and $Z_k = \xi(x_k) + Z- \xi(x) \nsd \xi(x_k)$.
\end{proof}

\subsection{Proof of Theorem \ref{thm:maincl1}}

The first part of the theorem is an immediate consequence of Theorem \ref{thm:clmatrix}, using the interior point $e=(I,I) \in \S^n_{++} \times \S^n_{++}$ and the homogeneity of $P_g$
\[
P_g( (1-\epsilon) (X,Y) + \epsilon (I,I) ) = (1-\epsilon) P_g(X+\frac{\epsilon}{1-\epsilon} I, Y+\frac{\epsilon}{1-\epsilon} I)
\]
which implies that $P_g( (1-\epsilon) (X,Y) + \epsilon (I,I) )$ has the same limit as $P_g(X+\epsilon I, Y + \epsilon I)$ when $\epsilon \downarrow 0$.

The second part of the theorem essentially follows from~\cite[Prop. 3.26--3.29]{hiai2017different}, where an explicit formula for the limits in terms of generalized inverses is also given. To make the paper self-contained, and because the proofs of the cited propositions are quite lengthy, we include proofs here as a sequence of short lemmas and corollaries.

\begin{lemma}
\label{lem:monotone1}
Let $g:(0,\infty)\to \RR$ be operator concave such that $g(0^+) > -\infty$. Then $(X,Y) \mapsto P_g(X,Y) - g(0^+) X$ is monotone in its first argument.
\end{lemma}
\begin{proof}
The integral representation in \cite[Eq. (2.3)]{hiai2017different} tells us that
\[ g(x) = g(0^+) + ax -bx^2 - \int_{(0,\infty)}\frac{x}{1+s} - \frac{1}{1+sx^{-1}}\;d\mu(s)\]
where $a\in \RR$ and $b\geq 0$ and $\int_{(0,\infty)}(1+s)^{-2}\;d\mu(s)<\infty$. Therefore 
	\[ P_g(X,Y) = g(0^+)X + aY - bYX^{-1}Y - \int_{(0,\infty)} \frac{1}{1+s}Y - (X^{-1}+sY^{-1})^{-1}\;d\mu(s).\]
	It follows that $P_g(X,Y) - g(0^+)X$ is monotone in its first argument.
\end{proof}

\begin{corollary}
\label{cor:caseg0finite}
Let $g:(0,\infty)\to \RR$ be operator concave with $g(0^+) > -\infty$. If $(X,Y) \in \S^n_+ \times \S^n_+$ satisfy $Y \ll X$, then $P_g(X+\epsilon I, Y + \epsilon I)$ is bounded below as $\epsilon \downarrow 0$.
\end{corollary}
\begin{proof}
Assume $Y \ll X$, i.e., there exists $c>0$ such that $X \psd cY$. By Lemma \ref{lem:monotone1}, $P_g(X,Y) - g(0^+)X$ is monotone in its first argument. Hence, 
\begin{equation}
\label{eq:lbPgXeI} 
P_g(X+\epsilon I, Y+\epsilon I)  - g(0^+)(X+\epsilon I)  \psd P_g(cY+\epsilon I, Y+\epsilon I) - g(0^+)(cY+\epsilon I).
\end{equation}
If $Y = \sum_{i=1}^n \lambda_i v_iv_i^*$ is the spectral decomposition of $Y$, then
\[
	cY+\eps I = \sum_{i=1}^n (c\lambda_i + \epsilon) v_i v_i^*\quad\textup{and}\quad
Y+\eps I = \sum_{i=1}^n (\lambda_i + \epsilon) v_i v_i^*
\]
so that $P_g(cY+\epsilon I, Y+\epsilon I) = \sum_{i=1}^n P_g(c\lambda_i + \epsilon, \lambda_i + \epsilon) v_i v_i^* = \sum_{i=1}^n (c\lambda_i + \epsilon)g(\frac{\lambda_i + \epsilon}{c\lambda_i + \epsilon}) v_i v_i^*$. The right hand side of \eqref{eq:lbPgXeI} thus satisfies
		\begin{align*}
			\sum_{i=1}^{n}[(c\lambda_i+\epsilon) g\left(\frac{\lambda_i+\epsilon}{c\lambda_i+\epsilon}\right) - g(0^+)(c\lambda_i+\epsilon)]v_iv_i^* \rightarrow c(g(1/c) - g(0^+))Y
		\end{align*}
as $\epsilon\downarrow0$. This completes the proof.
\end{proof}

\begin{corollary}
\label{cor:caseg0ghat0finite}
Let $g:(0,\infty)\to \RR$ be operator concave with $g(0^+) > -\infty$ and $\hat{g}(0^+) > -\infty$. If $(X,Y) \in \S^n_+ \times \S^n_+$, then $P_g(X+\epsilon I, Y+\epsilon I)$ is bounded below as $\epsilon \downarrow 0$.
\end{corollary}
\begin{proof}
By Lemma \ref{lem:monotone1} we know that $P_g(X,Y) - g(0^+) X$ is monotone in $X$. Applying Lemma \ref{lem:monotone1} to $\hat{g}$ which is operator concave, we get that $P_{\hat{g}}(Y,X) - \hat{g}(0^+) Y$ is monotone in $Y$. Using $P_g(X,Y) = P_{\hat g}(Y,X)$ we get that $P_g(X,Y) - g(0^+) X - \hat{g}(0^+) Y = P_h(X,Y)$ where $h(x) = g(x) - g(0^+) - \hat{g}(0^+) x$, is monotone in both arguments since, for any $A,B \psd 0$,
\[
\begin{aligned}
	P_g(X+A,Y+B) &- g(0^+)(X+A)-\hat g(0^+)(Y+B) \\
	&\psd P_g(X,Y+B) - g(0^+) X - \hat{g}(0^+)(Y+B)\\
&\psd P_g(X,Y)-g(0^+) X - \hat{g}(0^+) Y.
\end{aligned}
\]
To conclude, observe that
		\begin{align*}
			P_g(X+\epsilon I,Y+\epsilon I) & = P_h(X+\epsilon I,Y+\epsilon I) + g(0^+)(Y+\epsilon I) + \hat{g}(0^+)(X+\epsilon I)\\
			&\psd P_h(\epsilon I,\epsilon I) + g(0^+)(Y+\epsilon I) + \hat{g}(0^+)(X+\epsilon I)\\
		& = \epsilon h(1)I + g(0^+)(Y+\epsilon I) + \hat{g}(0^+)(X+\epsilon I)\\
			&\to g(0^+) Y + \hat{g}(0^+) X.
\end{align*}
\end{proof}

\begin{lemma}
\label{lem:caseg0ghat0infinite}
Let $g:(0,\infty)\to \RR$ be operator concave with $g(0^+) = \hat{g}(0^+) = -\infty$. If  $(X,Y)\in \S_{+}\times \S_{+}$ satisfy $\ker(X) = \ker(Y)$, then $P_g(X+\epsilon I, Y+\epsilon I)$ is bounded below as $\epsilon \downarrow 0$.
\end{lemma}
\begin{proof}
	Since $\ker(X)=\ker(Y)$ there is a unitary matrix $Q$ such that 
	\[ Q^*XQ = \begin{bmatrix}\tilde{X} & 0\\0 & 0\end{bmatrix}\quad\textup{and}\quad
		Q^*YQ = \begin{bmatrix} \tilde{Y} & 0\\0 & 0\end{bmatrix}\]
		where $\tilde{X}$ and $\tilde{Y}$ are positive definite. Then 
		\[ P_g(X+\epsilon I,Y+\epsilon I) = Q\begin{bmatrix} P_g(\tilde{X}+\epsilon I,\tilde{Y} + \epsilon I) & 0 \\0 & P_g(\epsilon I,\epsilon I)\end{bmatrix}Q^*\]
		 and so by continuity on the 
			positive definite cone,
			$\lim_{\epsilon \downarrow 0} P_g(\tilde{X}+\epsilon I,\tilde{Y} + \epsilon I) = P_g(\tilde{X},\tilde{Y})$. Furthermore, 
			$\lim_{\epsilon \downarrow 0} P_g(\epsilon I,\epsilon I) = \epsilon g(1)I = 0$.
\end{proof}

\begin{lemma}
\label{lem:unbounded33}
Let $g:(0,\infty)\to \RR$ be operator concave such that $g(0^+) = -\infty$. If $X \not\ll Y$ then $P_g(X+\epsilon I, Y+\epsilon I)$ is unbounded below as $\epsilon \downarrow 0$.\\
More generally, if $\phi:\S^n \to \S^m$ is a positive linear map such that \eqref{eq:assumptionphi} holds, then $\phi(P_g(X+\epsilon I, Y+\epsilon I))$ is unbounded below as $\epsilon\downarrow0$.
\end{lemma}
\begin{proof}
Assume $\ker(Y) \not\subset \ker(X)$. Let $v \in \CC^n$ with $\|v\|_2^2 = 1$, such that $x = v^* X v > 0$ and $v^* Y v = 0$. 
Since $g$ is operator concave, the \emph{operator Jensen inequality} says that for any $R \in \CC^{n\times m}$ such that $R^* R = I_m$ and $X \in \H^n_{++}$,
\[
g(R^*XR) \psd R^* g(X) R,
\]
see \cite{hansen2003jensen} and \cite[Exercise V.2.2(iii)]{bhatia2013matrix}. If we let $R = \frac{(X + \epsilon I)^{1/2} v}{\|(X + \epsilon I)^{1/2} v\|_2} = \frac{(X + \epsilon I)^{1/2} v}{\sqrt{x+\epsilon}} \in \CC^n$ which satisfies $R^* R = 1$ we get
\[
\begin{aligned}
v^* P_g(X+\epsilon I,Y+\epsilon I) v &= (x+\epsilon) R^* g\left((X+\epsilon I)^{-1/2} (Y+\epsilon I) (X+\epsilon I)^{-1/2}\right) R\\
&\leq (x+\epsilon) g(v^*(Y+\epsilon I) v / \|(X+\epsilon I) v\|_2^2)\\
&= (x+\epsilon) g\left(\frac{\epsilon}{x+\epsilon}\right) \to -\infty
\end{aligned}
\]
as $\epsilon \downarrow 0$.
This shows that
$\lim_{\epsilon\downarrow0} \tr P_g(X+\epsilon I, Y+\epsilon I) = -\infty$.

Assume now $\phi$ is a linear map such that \eqref{eq:assumptionphi} holds. We will show that
\[
\lim_{\epsilon\downarrow0} \tr \phi( P_g(X+\epsilon I, Y+\epsilon I)) = -\infty.
\]
Let $\psi(X) = \tr \phi(X) - c \tr X$. Assumption \eqref{eq:assumptionphi} says that $\psi$ is a positive map, and hence $\psi(P_g(X+\epsilon I, Y+\epsilon I))$ is concave in $\epsilon$ and thus has a limit in $\RR \cup \{-\infty\}$ as $\epsilon\downarrow0$. It follows that
\[
\tr \phi( P_g(X+\epsilon I, Y+\epsilon I)) = \psi(P_g(X+\epsilon I, Y+\epsilon I)) + c \tr P_g(X+\epsilon I, Y+\epsilon I) \to -\infty
\]
as $\epsilon \downarrow 0$ as desired.
\end{proof}

We can now complete the proof of \ref{thm:maincl1}. For (i) Use Corollary \ref{cor:caseg0ghat0finite}. For
(ii) Use Corollary \ref{cor:caseg0finite} with $\hat{g}$, and Lemma \ref{lem:unbounded33}. For
(iii) Use Corollary \ref{cor:caseg0finite} and Lemma \ref{lem:unbounded33}. For
(iv) Use Lemmas \ref{lem:caseg0ghat0infinite} and \ref{lem:unbounded33}.

\subsection{Proof of Theorem \ref{thm:maincl2}}
\label{app:Q}

The first part of the theorem is identical to the proof of Theorem \ref{thm:maincl1}.
We only focus on the second part. 

Note that when $g(0) > -\infty$ and $\hat{g}(0) > -\infty$, then $P_g$ is defined on all pairs of positive semidefinite matrices, and thus so is $Q_g$. This establishes (i).

When $g(0) = -\infty$, we cannot directly obtain the domain of $Q_g$ from that of $P_g$. Indeed, we have seen that in this case, the maximal domain of $P_g$ is $\{(X,Y) : X \ll Y\}$. Observe however that $X\otimes I \ll I \otimes Y$ requires $Y$ to be invertible. We thus need to study $Q_g$ directly. The aim of the next lemma is to 
identify the set
\[
\cD = \{(X,Y) \in \S^n_+ \times \S^n_+ : Q_g(X+\epsilon I|Y+\epsilon I) \text{ is bounded below as $\epsilon\downarrow0$} \}.\]
\begin{lemma}
	\label{lem:Qdom}
Let $g:(0,\infty) \to \RR$ be operator concave with $g(0) = -\infty$ and $\hat{g}(0) > -\infty$, and let $\p:\S^{n^2}\to \RR$ be the linear map such that $\p(X\otimes \bar Y) = \tr(XY)$ for all $(X,Y) \in (\H^n)^2$. Let $(X,Y) \in \S^n_{+} \times \S^n_+$. Then the limit $\lim_{\epsilon\downarrow0} Q_g(X+\epsilon I|Y+\epsilon I)$ is finite if, and only if, $X \ll Y$.
\end{lemma}
\begin{proof}
Let $X = \sum_{i} \lambda_i P_i$ and $Y = \sum_{j} \mu_j Q_j$ be spectral decompositions of $X$ and $Y$, where $P_i$ and $Q_j$ are orthogonal projectors on the respective eigenspaces of $X$ and $Y$ respectively. Note that
\[
\begin{aligned}
(I\otimes Y+\epsilon I)^{-1/2} ((X+\epsilon I)\otimes I) (I\otimes Y+\epsilon I)^{-1/2} &= (X+\epsilon I) \otimes (Y+\epsilon I)^{-1}\\
	&= \sum_{i,j} (\lambda_i + \epsilon) (\mu_j+\epsilon)^{-1} P_i \otimes Q_j.
\end{aligned}
\]
It thus follows that $Q_g(X+\epsilon I|Y+\epsilon I) = \p(P_g((X + \epsilon I)\otimes I, I \otimes (Y+\epsilon I))) = \p(P_{\hat g}(I\otimes (Y+\epsilon I), (X + \epsilon I) \otimes I))$ is given by
\[
Q_{g}(X+\epsilon I|Y+\epsilon I) = \sum_{ij} (\mu_j+\epsilon) \hat{g}\left( \frac{\lambda_i + \epsilon}{\mu_j+\epsilon} \right) \tr(P_i Q_j).
\]
Let $P^0,Q^0$ be respectively the projectors on $\ker(X)$ and $\ker(Y)$. If we decompose the sum above according to whether the eigenvalues are zero or not, we get:
\begin{align}
\label{eq:Qgdecomp-app}
	Q_{g}(X+\epsilon I|Y+\epsilon I) &= \sum_{ij:\mu_j > 0,\lambda_i \geq 0} (\mu_j+\epsilon) \hat{g}\left( \frac{\lambda_i + \epsilon}{\mu_j+\epsilon} \right) \tr(P_i Q_j)\\\nonumber
	&\qquad + \sum_{i:\lambda_i > 0} \epsilon \hat{g}\left(\frac{\lambda_i+\epsilon}{\epsilon}\right) \tr(P_i Q^0) +\epsilon \hat{g}(1) \tr(P^0 Q^0).
\end{align}
Now assume that $X\ll Y$. We see that if $\lambda_i > 0$, then $\tr(P_i Q^0) = 0$ since $\ker(Y) \subset \ker(X) = \textrm{im}(X)^{\perp}$. Thus we get in this case:
\[
Q_{g}(X+\epsilon I | Y+\epsilon I) = \sum_{ij:\mu_j > 0,\lambda_i \geq 0} (\mu_j+\epsilon) \hat{g}\left( \frac{\lambda_i+\epsilon}{\mu_j+\epsilon} \right) \tr(P_i Q_j) + \epsilon \hat{g}(1) \tr(P^0 Q^0).
\]
Letting $\epsilon\downarrow0$ we see that this has a finite limit since $\hat{g}(0^+) > -\infty$.

Conversely, assume that $X \not\ll Y$. Then in this case note that the middle terms of \eqref{eq:Qgdecomp-app} all diverge to $-\infty$ as $\epsilon\downarrow0$ since $\epsilon \hat{g}((\lambda_i+\epsilon)/\epsilon) = (\lambda_i+\epsilon) g(\epsilon/(\lambda_i+\epsilon)) \to -\infty$ as $\epsilon\downarrow0$, by our assumption on $g$.
\end{proof}

The lemma above establishes case (ii) of Theorem \ref{thm:maincl2}. Case (iii) is obtained by applying the lemma above with $\hat{g}$ instead of $g$. Similar arguments establish case (iv).

\bibliographystyle{alpha}
\bibliography{sc}

\newcommand{\etalchar}[1]{$^{#1}$}
\begin{thebibliography}{MLDS{\etalchar{+}}13}

\bibitem[And79]{ando1979concavity}
Tsuyoshi Ando.
\newblock Concavity of certain maps on positive definite matrices and
  applications to {H}adamard products.
\newblock {\em Linear Algebra Appl.}, 26:203--241, 1979.

\bibitem[Bac22]{bach2022information}
Francis Bach.
\newblock Information theory with kernel methods.
\newblock {\em IEEE Transactions on Information Theory}, 2022.

\bibitem[Bar14]{bartle2014elements}
Robert~G Bartle.
\newblock {\em The elements of integration and Lebesgue measure}.
\newblock John Wiley \& Sons, 2014.

\bibitem[Bei13]{beigi2013sandwiched}
Salman Beigi.
\newblock Sandwiched {R}{\'e}nyi divergence satisfies data processing
  inequality.
\newblock {\em J. Math. Phys.}, 54(12):122202, 2013.

\bibitem[Bha13]{bhatia2013matrix}
Rajendra Bhatia.
\newblock {\em Matrix analysis}, volume 169.
\newblock Springer Science \& Business Media, 2013.

\bibitem[BS82]{bsentropy}
Viacheslav~P Belavkin and P~Staszewski.
\newblock {$C^*$-algebraic generalization of relative entropy and entropy}.
\newblock In {\em Annales de l'IHP Physique th{\'e}orique}, volume~37, pages
  51--58, 1982.

\bibitem[CKV22a]{coey2022conic}
C.~Coey, L.~Kapelevich, and J.~P. Vielma.
\newblock Conic optimization with spectral functions on {E}uclidean {J}ordan
  algebras.
\newblock {\em Mathematics of Operations Research}, 2022.

\bibitem[CKV22b]{coey2022performance}
C.~Coey, L.~Kapelevich, and J.~P. Vielma.
\newblock Performance enhancements for a generic conic interior point
  algorithm.
\newblock {\em Math. Prog. Comput.}, pages 1--49, 2022.

\bibitem[CS17]{chandrasekaran2017relative}
Venkat Chandrasekaran and Parikshit Shah.
\newblock Relative entropy optimization and its applications.
\newblock {\em Math. Program.}, 161(1):1--32, 2017.

\bibitem[Eff09]{effros2009matrix}
Edward~G. Effros.
\newblock A matrix convexity approach to some celebrated quantum inequalities.
\newblock {\em Proc. Natl. Acad. Sci. USA}, 106(4):1006--1008, 2009.

\bibitem[ENG11]{ebadian2011perspectives}
Ali Ebadian, Ismail Nikoufar, and Madjid~Eshaghi Gordji.
\newblock Perspectives of matrix convex functions.
\newblock {\em Proc. Natl. Acad. Sci. USA}, 108(18):7313--7314, 2011.

\bibitem[FF21a]{fang2021geometric}
Kun Fang and Hamza Fawzi.
\newblock Geometric {R}{\'e}nyi divergence and its applications in quantum
  channel capacities.
\newblock {\em Comm. Math. Phys.}, 384(3):1615--1677, 2021.

\bibitem[FF21b]{fawzi2021defining}
Hamza Fawzi and Omar Fawzi.
\newblock Defining quantum divergences via convex optimization.
\newblock {\em Quantum}, 5:387, 2021.

\bibitem[FGP{\etalchar{+}}22]{fawzi2022lifting}
H.~Fawzi, J.~Gouveia, P.~A. Parrilo, J.~Saunderson, and R.~R. Thomas.
\newblock Lifting for simplicity: Concise descriptions of convex sets.
\newblock {\em SIAM Rev.}, 64(4):866--918, 2022.

\bibitem[FL13]{frank2013monotonicity}
Rupert~L Frank and Elliott~H Lieb.
\newblock Monotonicity of a relative {R}{\'e}nyi entropy.
\newblock {\em J. Math. Phys.}, 54(12):122201, 2013.

\bibitem[FSP19]{logapprox}
Hamza Fawzi, James Saunderson, and Pablo~A Parrilo.
\newblock Semidefinite approximations of the matrix logarithm.
\newblock {\em Found. Comput. Math.}, 19(2):259--296, 2019.

\bibitem[FT17]{faybusovich2017matrix}
Leonid Faybusovich and Takashi Tsuchiya.
\newblock Matrix monotonicity and self-concordance: how to handle quantum
  entropy in optimization problems.
\newblock {\em Optim.\ Lett.}, 11(8):1513--1526, 2017.

\bibitem[Fuj89]{fujii1989relative}
JI~Fujii.
\newblock Relative operator entropy in noncommutative information theory.
\newblock {\em Math. Japon.}, 34:341--348, 1989.

\bibitem[FZ20]{faybusovich2020self}
Leonid Faybusovich and Cunlu Zhou.
\newblock Self-concordance and matrix monotonicity with applications to quantum
  entanglement problems.
\newblock {\em Appl.\ Math.\ Comput.}, 375:125071, 2020.

\bibitem[Hia10]{hiainotes}
Fumio Hiai.
\newblock Matrix analysis: matrix monotone functions, matrix means, and
  majorization.
\newblock {\em Interdisciplinary Information Sciences}, 16(2):139--248, 2010.
\newblock Available online at
  \url{https://www.jstage.jst.go.jp/article/iis/16/2/16_2_139/_article/-char/en}.

\bibitem[HIL{\etalchar{+}}21]{hu2021robust}
Hao Hu, Jiyoung Im, Jie Lin, Norbert L{\"u}tkenhaus, and Henry Wolkowicz.
\newblock Robust interior point method for quantum key distribution rate
  computation.
\newblock {\em arXiv preprint arXiv:2104.03847}, 2021.

\bibitem[HM17]{hiai2017different}
Fumio Hiai and Mil{\'a}n Mosonyi.
\newblock Different quantum $f$-divergences and the reversibility of quantum
  operations.
\newblock {\em Rev. Math. Phys.}, 29(07):1750023, 2017.

\bibitem[HP03]{hansen2003jensen}
Frank Hansen and Gert~K Pedersen.
\newblock Jensen's operator inequality.
\newblock {\em Bull. Lond. Math. Soc.}, 35(4):553--564, 2003.

\bibitem[KSD09]{kulis2009low}
Brian Kulis, M{\'a}ty{\'a}s~A Sustik, and Inderjit~S Dhillon.
\newblock Low-rank kernel learning with {B}regman matrix divergences.
\newblock {\em J. Mach. Learn. Res.}, 10(2), 2009.

\bibitem[KT19]{karimi2019domain}
Mehdi Karimi and Levent Tun{\c{c}}el.
\newblock Domain-driven solver ({DDS}) version 2.0: a {MATLAB}-based software
  package for convex optimization problems in domain-driven form.
\newblock {\em arXiv preprint arXiv:1908.03075}, 2019.

\bibitem[KT20]{karimi2020primal}
Mehdi Karimi and Levent Tun{\c{c}}el.
\newblock Primal--dual interior-point methods for domain-driven formulations.
\newblock {\em Math. Oper. Res.}, 45(2):591--621, 2020.

\bibitem[Lie73]{lieb}
Elliott~H. Lieb.
\newblock Convex trace functions and the {W}igner-{Y}anase-{D}yson conjecture.
\newblock {\em Adv. Math.}, 11(3):267--288, 1973.

\bibitem[LLP22]{lindstrom2020error}
S.~B. Lindstrom, B.~F. Louren{\c{c}}o, and T.~K. Pong.
\newblock Error bounds, facial residual functions and applications to the
  exponential cone.
\newblock {\em Mathematical Programming}, pages 1--50, 2022.

\bibitem[LR73]{liebruskai}
Elliott~H. Lieb and Mary~Beth Ruskai.
\newblock Proof of the strong subadditivity of quantum mechanical entropy.
\newblock {\em J. Math. Phys.}, 14(12):1938--1941, 1973.

\bibitem[Mat15]{matsumoto}
Keiji Matsumoto.
\newblock A new quantum version of $f$-divergence.
\newblock In {\em Nagoya Winter Workshop: Reality and Measurement in Algebraic
  Quantum Theory}, pages 229--273. Springer, 2015.

\bibitem[MLDS{\etalchar{+}}13]{MDSFT13}
M.~M{\"u}ller-Lennert, F.~Dupuis, O.~Szehr, S.~Fehr, and M.~Tomamichel.
\newblock {On quantum {R}{\'e}nyi entropies: A new generalization and some
  properties}.
\newblock {\em J. Math. Phys.}, 54(12):122203, 2013.

\bibitem[Nes18]{nesterov2018lectures}
Yurii Nesterov.
\newblock {\em Lectures on convex optimization}, volume 137 of {\em Springer
  Optimization and Its Applications}.
\newblock Springer, 2018.

\bibitem[NN94]{nesterovnemirovski}
Yurii Nesterov and Arkadii Nemirovskii.
\newblock {\em Interior-point polynomial algorithms in convex programming}.
\newblock SIAM, 1994.

\bibitem[NT98]{nesterovtodd}
Yu~E Nesterov and Michael~J Todd.
\newblock Primal-dual interior-point methods for self-scaled cones.
\newblock {\em SIAM J. Optim.}, 8(2):324--364, 1998.

\bibitem[Pet86]{petz1986quasi}
D{\'e}nes Petz.
\newblock Quasi-entropies for finite quantum systems.
\newblock {\em Rep. Math. Phys.}, 23(1):57--65, 1986.

\bibitem[PY22]{papp2022alfonso}
D.~Papp and S.~Y{\i}ld{\i}z.
\newblock alfonso: {M}atlab package for nonsymmetric conic optimization.
\newblock {\em INFORMS Journal on Computing}, 34(1):11--19, 2022.

\bibitem[RISB20]{quantumblahutarimoto}
Navneeth Ramakrishnan, Raban Iten, Volkher Scholz, and Mario Berta.
\newblock Quantum {B}lahut-{A}rimoto algorithms.
\newblock In {\em Proc.\ IEEE International Symposium on Information Theory
  (ISIT)}, pages 1909--1914. IEEE, 2020.

\bibitem[Roc70]{rockafellar}
R~Tyrrell Rockafellar.
\newblock {\em Convex Analysis}, volume~36.
\newblock Princeton University Press, 1970.

\bibitem[SSER15]{sutter2015efficient}
David Sutter, Tobias Sutter, Peyman~Mohajerin Esfahani, and Renato Renner.
\newblock Efficient approximation of quantum channel capacities.
\newblock {\em IEEE Transactions on Information Theory}, 62(1):578--598, 2015.

\bibitem[Tom15]{tomamichel2015quantum}
Marco Tomamichel.
\newblock {\em Quantum information processing with finite resources:
  mathematical foundations}, volume~5.
\newblock Springer, 2015.

\bibitem[Uch10]{uchiyama2010operator}
Mitsuru Uchiyama.
\newblock Operator monotone functions, positive definite kernels and
  majorization.
\newblock {\em Proc. Amer. Math. Soc.}, 138(11):3985--3996, 2010.

\bibitem[WLC18]{winick2018reliable}
Adam Winick, Norbert L{\"u}tkenhaus, and Patrick~J Coles.
\newblock Reliable numerical key rates for quantum key distribution.
\newblock {\em Quantum}, 2:77, 2018.

\bibitem[WWY14]{WWY13}
M.M. Wilde, A.~Winter, and D.~Yang.
\newblock Strong converse for the classical capacity of entanglement-breaking
  and {H}adamard channels via a sandwiched {R}{\'e}nyi relative entropy.
\newblock {\em Comm. Math. Phys.}, 331(2):593--622, 2014.

\bibitem[YCL22]{you2021minimizing}
J.-K. You, H.-C. Cheng, and Y.-H. Li.
\newblock Minimizing quantum {R}{\'e}nyi divergences via mirror descent with
  {P}olyak step size.
\newblock In {\em 2022 IEEE International Symposium on Information Theory
  (ISIT)}, pages 252--257. IEEE, 2022.

\bibitem[ZFG10]{zinchenko2010numerical}
Yuriy Zinchenko, Shmuel Friedland, and Gilad Gour.
\newblock Numerical estimation of the relative entropy of entanglement.
\newblock {\em Phys. Rev. A}, 82(5):052336, 2010.

\bibitem[Zha20]{zhang2020wigner}
Haonan Zhang.
\newblock From {Wigner-Yanase-Dyson} conjecture to {Carlen-Frank-Lieb}
  conjecture.
\newblock {\em Adv.\ Math.}, 365:107053, 2020.

\end{thebibliography}

\end{document}